\documentclass[12pt,A4]{article}
\usepackage[applemac]{inputenc}
\usepackage[english]{babel}
\usepackage{amsmath,amsfonts,amssymb,amsthm,mathrsfs,bbm,mathtools}
\usepackage[font=sf, labelfont={sf,bf}, margin=1cm]{caption}
\usepackage{graphicx,graphics}
\usepackage{epsfig}
\usepackage{latexsym}
 \usepackage{ae,aecompl}
\usepackage{pstricks}
\usepackage{enumerate}
\usepackage{xcolor}
\usepackage{pifont}
\usepackage[pdfpagemode=UseNone,bookmarksopen=false,colorlinks=true,urlcolor=blue,citecolor=blue,citebordercolor=blue,linkcolor=blue]{hyperref}
\usepackage[top=2.5cm,left=1.5cm,right=1.5cm,bottom=2.5cm]{geometry}
\usepackage{comment}
\usepackage{mathpazo} 
\usepackage{eulervm}
\DeclareGraphicsExtensions{.jpg,.pdf}
\usepackage[normalem]{ulem}
\usepackage{enumerate,framed}

\newtheorem{thm}{Theorem}
\newtheorem{lem}{Lemma}[section]
\newtheorem{prop}[lem]{Proposition}
\newtheorem{cor}[lem]{Corollary}

\theoremstyle{definition}

\newtheorem{rem}[lem]{Remark}

\newtheorem{ques}[lem]{Question}

\newcommand{\E}[1]{\mathbb{E}  [ #1  ]}
\newcommand{\R}{\mathbb R}
\newcommand{\N}{\mathbb N}
\newcommand{\Z}{\mathbb Z}

\newcommand{\D}{\mathbb D}
\def \P {\mathbb{P}}

\def\dsk{d_{\mathrm{SK}}}

\newcommand\Es[1]{\mathbb{E}\left[#1\right]}
\renewcommand\Pr[1]{\mathbb{P}\left(#1\right)}
\newcommand \fl[1] {\left\lfloor #1 \right\rfloor}
\newcommand \ce[1] {\left\lceil #1 \right\rceil}

\newcommand{\norme}[1]{\left\Vert #1\right\Vert _ \infty}

\title{  \vspace {-2.2cm}\textbf{Self-similar scaling limits of Markov chains on the positive integers}}
\date{}
\author{}

\DeclareSymbolFont{extraup}{U}{zavm}{m}{n}
\DeclareMathSymbol{\varheart}{\mathalpha}{extraup}{86}
\DeclareMathSymbol{\vardiamond}{\mathalpha}{extraup}{87}

\makeatletter
\renewcommand*{\@fnsymbol}[1]{\ensuremath{\ifcase#1\or  \spadesuit \or \varheart\or \vardiamond \or \clubsuit \or
   \mathsection\or \mathparagraph\or \|\or **\or \dagger\dagger
   \or \ddagger\ddagger \else\@ctrerr\fi}}
\makeatother

\author{Jean Bertoin\thanks{Universit\"at Z\"urich. \hfill  \texttt{jean.bertoin@math.uzh.ch}} 
\qquad \& \qquad Igor Kortchemski\thanks{Universit\"at Z\"urich.\hfill  \texttt{igor.kortchemski@normalesup.org}} 
}

\begin{document}

\maketitle

\let\thefootnote\relax\footnotetext{ \\\emph{MSC2010 subject classifications}. Primary 60F17,60J10,60G18; secondary 60J35. }
 
\vspace {-1cm}

\begin{abstract} 
We are interested in the asymptotic behavior of Markov chains on the set of positive integers for which, loosely speaking, large jumps are rare and occur at a rate that behaves like a negative power of the current state, and such that  small positive and negative steps of the chain roughly compensate each other. If $X_{n}$ is such a Markov chain started at $n$, we establish a limit theorem for $\frac{1}{n}X_{n}$ appropriately scaled in time, where the scaling limit is given by a nonnegative self-similar Markov process. We also study the asymptotic behavior of the time needed by $X_{n}$ to reach some fixed finite set. We identify three different regimes  (roughly speaking the transient, the recurrent and the positive-recurrent regimes) in which $X_{n}$ exhibits different behavior. The present results extend those of Haas \& Miermont \cite {HM11} who focused on the case of non-increasing Markov chains. We further present a number of applications to the study of Markov chains with asymptotically zero drifts such as Bessel-type random walks, nonnegative self-similar Markov processes, invariance principles for random walks conditioned to stay positive, and exchangeable coalescence-fragmentation processes.
\end{abstract}

 \vfill

 \begin{figure}[!h]
 \begin{center}
    \includegraphics[width=0.3 \linewidth]{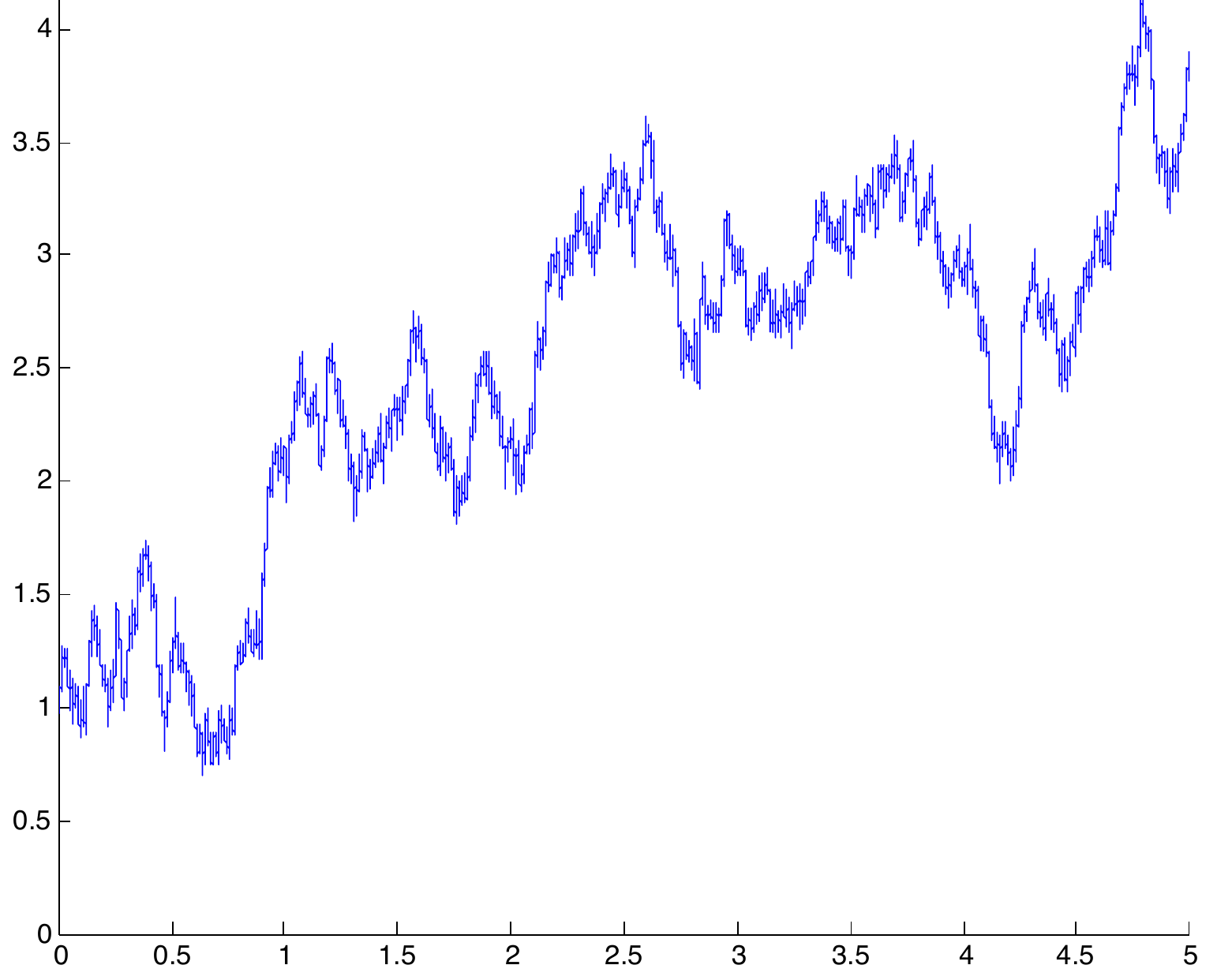}\hfill
  \includegraphics[width=0.3 \linewidth]{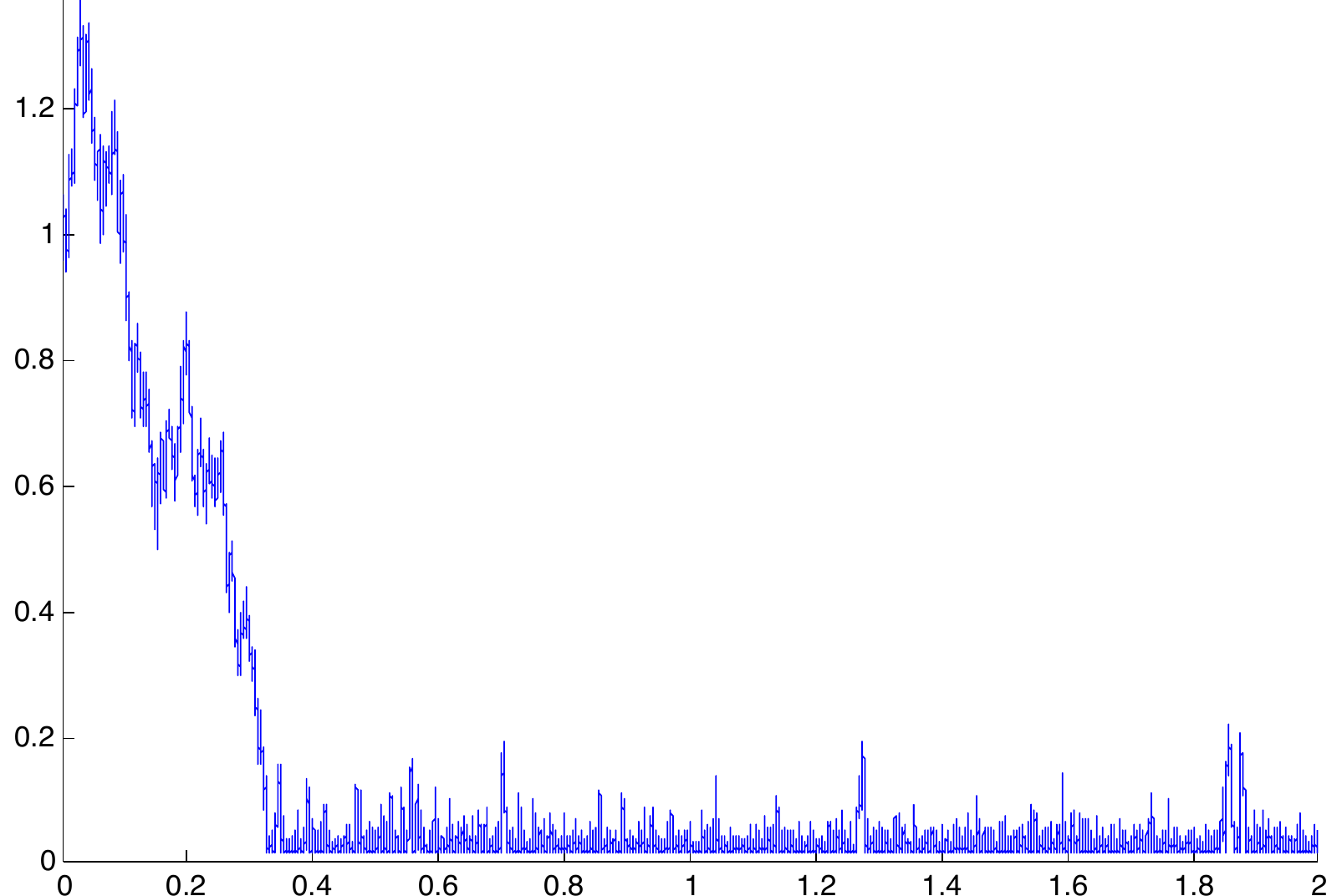}\hfill
\includegraphics[width=0.3 \linewidth]{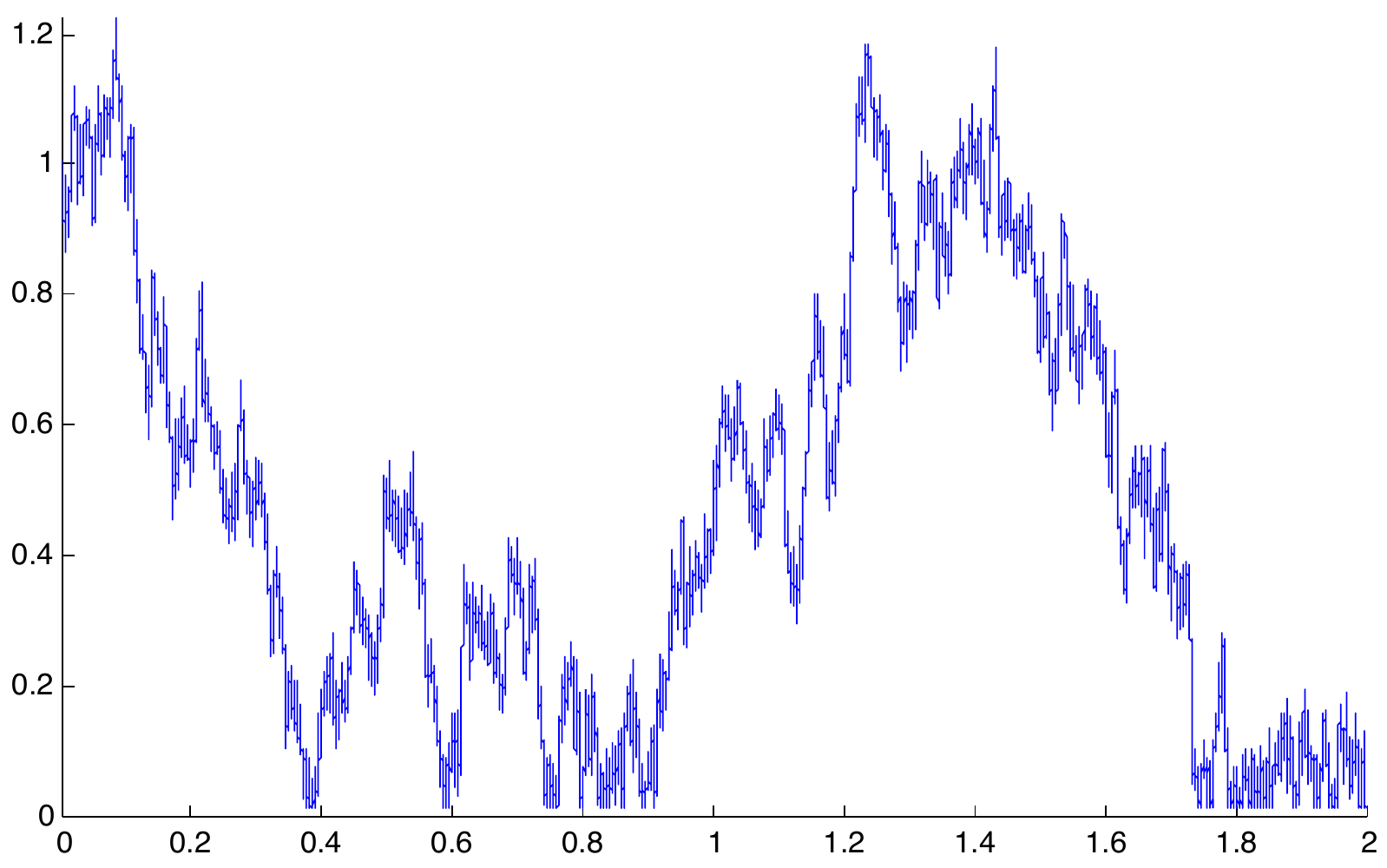}
 \caption{ \label{fig:MC} Three different asymptotic regimes of the Markov chain $X_{n}/n$ started at $n$ as $n \rightarrow \infty$: with probability tending to one as $n \rightarrow \infty$, in the first case, the chain never reaches the boundary {(transient case)}; in the second case $X_{n}$ reaches the boundary and then stays within its vicinity  {on long time scales} {(positive recurrent case)}, and in the last case $X_{n}$ visits the boundary infinitely many times and makes some macroscopic excursions in between (null-recurrent case).}
 \end{center}
 \end{figure}
 
 \vfill
 
 \section{Introduction}

In short, the purpose of this work is to provide explicit criteria for the functional weak convergence of properly rescaled Markov chains on $\N=\{1,2,\ldots\}$.
Since it is well-known from the  work of Lamperti \cite{Lam62a} that self-similar processes arise as the scaling limit of general stochastic processes, and since in the case of Markov chains, one naturally expects  the Markov property to be preserved after convergence, scaling limits of rescaled Markov chains on $\N$ should thus belong to the class of 
self-similar Markov processes on $[0,\infty)$. The latter have been also introduced by Lamperti \cite{Lam72}, who pointed out a remarkable connexion with real-valued L\'evy processes which we shall recall later on. 
Considering the powerful arsenal of techniques which are nowadays available for establishing convergence in distribution for sequences of Markov processes (see in particular Ethier \& Kurtz \cite{EK86} and Jacod \& Shiryaev \cite{JS03}), it seems that the study of scaling limits of general Markov chains on $\N$ should be part of the folklore. Roughly speaking, it is well-known that weak convergence of Feller processes amounts to the convergence of infinitesimal generators (in some appropriate sense), and  the path should thus be essentially well-paved. 

However, there is a major obstacle for this natural approach. Namely, there is a delicate issue regarding the boundary of self-similar Markov processes on $[0,\infty)$: in some cases, $0$ is an absorbing boundary, in some other, $0$ is an entrance boundary, and further $0$ can also be a reflecting boundary, where the reflection can be either continuous or by a jump. See \cite{BS11, CKPR12, Fit06, Riv05, Riv07} and the references therein. Analytically, this raises the questions of identifying a core for a self-similar Markov process on $[0,\infty)$
and of determining its infinitesimal generator on this core,  in particular on the neighborhood of  the boundary point $0$ where a singularity appears. To the best of our knowledge, these questions remain open in general, and investigating the asymptotic behavior of a sequence of infinitesimal generators at a singular point therefore seems rather subtle.

A few years ago, Haas \& Miermont \cite{HM11} obtained a general scaling limit theorem for non-increasing Markov chains on $\N$ (observe that plainly, $1$ is always an absorbing boundary for non-increasing self-similar Markov processes), and the purpose of the present work is to extend their result by removing the non-increase assumption. Our approach bears similarities with that developed by Haas \& Miermont, but also with some differences. In short, Haas and Miermont first established a tightness result, and then analyzed weak limits of convergent subsequences via martingale problems, whereas we rather investigate asymptotics of infinitesimal generators. 

More precisely, in order to circumvent the crucial difficulty related to the boundary point $0$, we shall not directly study the rescaled version of the Markov chain, but rather of a time-changed version. The time-substitution is chosen so to yield weak convergence  towards the exponential of a L\'evy process, where the convergence is established through the analysis of infinitesimal generators. The upshot is that cores and infinitesimal generators are much better understood for L\'evy processes and their exponentials than for self-similar Markov processes, and  boundaries yield no difficulty. We are then left with the inversion of the time-substitution, and this turns out to be closely related to the Lamperti transformation. However, although  our approach enables us to treat the situation when the Markov chain is either absorbed at the boundary point $1$ or eventually escapes to $+\infty$, it does not seem to provide direct access to the case when the limiting process is reflected at the boundary (see Figure \ref{fig:MC}).

The rest of this work is organized as follows. Our general results are presented in Section \ref {sec:description}. 
We state three main limit theorems, namely Theorems \ref{thm:main1}, \ref{thm:main2} and \ref{thm:main3}, each being valid under some specific set of assumptions. Roughly speaking, Theorem \ref{thm:main1} treats the situation where the Markov chain is transient and thus escapes to $+\infty$, whereas Theorem \ref{thm:main2} deals with the recurrent case. In the latter, we only consider the Markov chain until its first entrance time in some finite set, which forces absorption at the boundary point $0$ for the scaling limit. Theorem \ref{thm:main3} is concerned with the situation where the Markov chain is positive recurrent; then convergence of the properly rescaled chain to a self-similar Markov process absorbed at $0$ is established, even though the Markov chain is no longer trapped in some finite set.  
Finally, we also provide a weak limit theorem (Theorem \ref{thm:absorption})  in the recurrent situation  for the first instant when the Markov chain started from a large level enters some fixed finite set. 
Section \ref {sec:aux} prepares the proofs of the preceding results, by focusing on an auxiliary continuous-time Markov chain which is both closely related to the genuine discrete-time Markov chain and easier to study. The connexion between the two relies on a Lamperti-type transformation. The proofs of the statements made in Section \ref {sec:description} are then given in Section \ref {sec:scaling} by analyzing the time-substitution; classical arguments relying on the celebrated Foster criterion for recurrence of Markov chains also play a crucial role.
We illustrate  our general results in Section \ref {sec:applications}. First, we check that they encompass those of Haas \& Miermont in the case where the chain is non-increasing. Then we derive functional limit theorems for Markov chains with asymptotically zero drift (this includes the so-called Bessel-type random walks which have been considered by many authors in the literature), scaling limits are then given in terms of Bessel processes.  Lastly, we derive a weak limit theorem for the number of particles in a fragmentation-coagulation process, of a type similar to that introduced by J. Berestycki \cite{Ber04}. Finally, in Section 6, we point at a series of open questions related to this work.

We conclude this Introduction by mentioning that our initial motivation for establishing such scaling limits for Markov chains on $\N$ was a question raised by Nicolas Curien concerning the study of random planar triangulations and their connexions with compensated fragmentations which has been developed in a subsequent work \cite{BCK15}.

\paragraph {Acknowledgments.} We thank an anonymous referee and Vitali Wachtel for several useful comments. I.K. would also like to thank Leif
D\"oring for stimulating discussions.

\section {Description of the main results}
 \label {sec:description}
 
For every integer $n \geq 1$, let $(p_{n,k} ; k \geq 1)$ be a sequence of non-negative real numbers such that $ \sum_{k \geq 1} p_{n,k}=1$, and let $(X_{n}(k) ; k \geq 0)$ be the discrete-time homogeneous Markov chain started at state $n$ such that the probability transition from state $i$ to state $j$ is $p_{i,j}$ for $i,j \in \N$. Specifically, $X_{n}(0)=n$, and  $ \Pr {X_{n}(k+1)=j \ | \ X_{n}(k)=i}=p_{i,j}$ for every $i,j \geq 1$ and $k \geq 0$. 
Under certain assumptions on the probability transitions, we establish (Theorems \ref {thm:main1}, \ref {thm:main2} and \ref {thm:main3} below) a functional invariance principle for $ \frac{1}{n}{X} _{n}$, appropriately scaled in time, to a nonnegative self-similar Markov process in the Skorokhod topology for c\`adl\`ag functions. In order to state our results, we first need to formulate the main assumptions.

\subsection {Main Assumptions}
\label {sec:main}

For $n \geq 1$, denote by $ \Pi_{n}^{\ast}$ the probability measure on $ \R$ defined  by
$$ \Pi_{n}^{\ast}(dx)= \sum_{k \geq 1}  p_{n,k} \cdot \delta_{\ln(k)-\ln(n)}(dx),$$ 
which is the law of $ \ln(X_{n}(1)/n)$. 
Let $(a_{n})_{n \geq 0}$ be a sequence of positive real numbers with regular variation of index $ \gamma > 0$, meaning that $a_{\lfloor xn\rfloor}/a_{n} \rightarrow x^{\gamma}$ as $n \rightarrow \infty$ for every fixed $ x>0$, where $ \fl {x}$ stands for the integer part of a real number $x$. Let $ \Pi$ be a measure on $ \R \backslash \{ 0\}$ such that  $ \Pi( \{ -1,1\})=0$ and 
\begin{equation}
\label{eq:condPi} \int_{- \infty}^{\infty} (1 \wedge x^{2})  \ \Pi(dx) < \infty.
\end{equation}
We require that $ \Pi( \{ -1,1\})=0$ for the sake of simplicity only, and it would be possible to treat the general case with mild modifications which are left to the reader. We also mention that some of our results could be extended to the case where $ \gamma=0$ and $a_{n} \rightarrow \infty$, but we shall not pursue this goal here.
Finally, denote by $ \overline{\R}=[-\infty,\infty]$ the extended real line. 

We now introduce our main assumptions:

\begin{framed}\textbf{(A1).} As $n \rightarrow \infty$, we have the following vague convergence of measures on $ \overline {\R} \backslash \{ 0\}$:
$$a_{n} \cdot \Pi^{\ast}_{n}(dx)  \quad\mathop{\longrightarrow}^{(v)}_{n \rightarrow \infty} \quad  \Pi(dx).$$\end {framed}
Or, in other words, we assume that
$$ a_{n} \cdot  \Es{f \left( \frac{X_n(1)}{n} \right) }  \quad\mathop{\longrightarrow}_{n \rightarrow \infty} \quad \int_{ \R} f(e^x) \ \Pi(dx)$$ for every continuous function $f$ with compact support in $ [0,\infty]\backslash \{ 1\}$.

\begin{framed}
\textbf{(A2).} The following two convergences holds:
$$ a_{n} \cdot  \int_{-1}^{1}x \ \Pi_{n}^{\ast}(dx)  \quad\mathop{\longrightarrow}_{n \rightarrow \infty} \quad b, \qquad a_{n} \cdot \int_{-1}^{1}x^{2} \ \Pi_{n}^{\ast}(dx)  \quad\mathop{\longrightarrow}_{n \rightarrow \infty} \quad \sigma^{2}+ \int_{-1}^{1} x^{2} \ \Pi(dx),$$
for some $ b \in \R $ and $ \sigma ^{2} \geq 0$.
\end {framed}

It is important to note that under \textbf{(A1)}, we may have $ \int_{-1}^{1}|x| \ \Pi(dx) = \infty$, in which case \textbf{(A2)} requires small positive and negative steps of the chain to roughly compensate each other. 

\subsection {Description of the distributional limit}
\label {sec:descr}

We now introduce several additional tools in order to describe the scaling limit of the Markov chain $X_{n}$. Let  $ (\xi(t))_{t \geq 0}$ be a L\'evy process with characteristic  exponent
given by the L\'evy--Khintchine formula 
$$ \Phi(\lambda)=  -\frac{1}{2} \sigma^{2} \lambda^{2} +  ib \lambda+ \int_{- \infty}^{\infty} \left( e^{i\lambda x}-1-i \lambda x \mathbbm {1}_{|x| \leq 1} \right)  \ \Pi(dx), \qquad \lambda \in \R.$$
Specifically, there is the identity $ \Es {e^{ i \lambda \xi(t)}}=e^{t \Phi( \lambda)}$ for $ t \geq 0, \lambda \in \R$. Then set $$I_{ \infty}= \int_{0}^{\infty} e^{ \gamma \xi(s)} \ ds \quad \in \quad (0, \infty].$$
It is known that $I_{\infty}<\infty$ a.s.~if  $\xi$ drifts to $- \infty$ (i.e.~$\lim_{t\to \infty} \xi(t)=-\infty$ a.s.), and 
$ I_{ \infty}= \infty$ a.s.~if $ \xi$ drifts to $+ \infty$ or oscillates (see e.g.~\cite[Theorem 1]{BY05} which also gives necessary and sufficient conditions involving $ \Pi$). Then for every $t \geq 0$, set
$$\tau(t)= \inf  \left\{ u \geq 0 ; \int_{0}^{u}  e^{ \gamma \xi(s)}ds>t \right\} $$
with the usual convention  $ \inf \emptyset = \infty$. Finally, define the Lamperti transform  \cite {Lam72} of $ \xi$ by
$$Y(t)= e^{ \xi(\tau(t))}  \quad \textrm { for }  \quad 0 \leq t <  I_{ \infty}, \qquad  \qquad Y(t)= 0 \quad \textrm { for }  \quad t \geq I_{\infty}.$$
In view of the preceding observations, $Y$ hits $0$ in finite time almost surely if, and only if, $ \xi$ drifts to $- \infty$.

By construction, the process $Y$ is a self-similar Markov process of index $1/ \gamma$ started at $1$. Recall that if $ \P_{x}$ is the law of a nonnegative Markov process $(M_{t})_{t \geq 0}$ started at $x \geq 0$, then $M$ is self-similar with index $ \alpha>0$ if the law of $(r^{-\alpha} M_{rt})_{t \geq 0}$ under $ \P_{x}$ is $\P_{r^{-\alpha}x}$ for every $r>0$ and $x \geq 0$. Lamperti \cite {Lam72} introduced and studied nonnegative self-similar Markov processes and established that, conversely, any self-similar Markov process which either never reaches the boundary states $0$ and $ \infty$, or reaches them continuously (in other words, there is no killing inside $(0, \infty$)) can be constructed by using the previous transformation.

\subsection {Invariance principle for $X_{n}$}

We are now ready to state our first main result, which is a limit theorem in distribution in the space of real-valued c\`adl\`ag functions $ \D(\R_{+},\R)$ on $ \R_{+}$ equipped with the $J_{1}$-Skorokhod topology (we refer to \cite[Chapter VI]{JS03} for background on the Skorokhod topology).

\begin {thm}[Transient case]\label {thm:main1} Assume that $ \textbf{(A1)}$ and \textbf{(A2)} hold, and that the L\'evy process $ \xi$ does not drift to $- \infty$. Then the convergence
\begin{equation}
\label{eq:cvthm1} \left( \frac{ {X} _{n}( \lfloor a_{n} t \rfloor)}{n} ; t \geq 0 \right)   \quad\mathop{\longrightarrow}^{(d)}_{n \rightarrow \infty} \quad (Y(t); t\geq 0)
\end{equation}
holds in distribution  in $ \D(\R_{+},\R)$.
\end {thm}

In this case, $Y$ does not touch $0$ almost surely (see the left-most image in Figure  \ref{fig:MC}). When $ \xi$ drifts to $- \infty$, we  establish an analogous result for the chain $X_{n}$ stopped when it reaches some fixed finite set  under the following additional assumption: 

\begin{framed}
\textbf{(A3).} There exists $ \beta>0$ such that $$\displaystyle \limsup_{n \rightarrow \infty} a_{n} \cdot \int_{1}^{\infty}  e^{\beta x }\ \Pi^{\ast}_{n}(dx) < \infty.$$
\end {framed}

 Observe that  \textbf{(A1)} and \textbf{(A3)} imply that $ \int_{1}^{\infty} e^{\beta x} \ \Pi(dx)< \infty$. Roughly speaking, Assumption \textbf{(A3)} tells us that in the case where $ \xi$ drifts to $-\infty$, the chain $X_{n}/n$ does not make too large positive jumps and will enable us to use Foster--Lyapounov type estimates (see Sec.~\ref {sec:FL}). Observe that \textbf{(A3)} is automatically satisfied if the Markov chain is non-increasing or has uniformly bounded upwards jumps.
 
 In the sequel, we let $K \geq 1$ be any fixed  integer such that the set $ \{ 1,2, \ldots,K\}$ is accessible by $X_{n}$ for every $n \geq 1$ (meaning that $ \inf \{ i \geq 0; X_{n}(i) \leq K \}< \infty$ with positive probability for every $n \geq 1$). It is a simple matter to check that if \textbf{(A1)}, \textbf{(A2)} hold and $ \xi$ drifts to $-\infty$, then such integers always exist. Indeed, consider
$$ \kappa \ \coloneqq \ \sup \left\{ n \geq 1: \P(X_n<n)=0  \right \}.$$ 
 If $\kappa= \infty$, then the measure $\Pi^\ast_n$ has support in $[0,\infty)$ for infinitely many $n \in \N$, and thus, if further \textbf{(A1)} and \textbf{(A2)} hold, $ \xi$ must be a subordinator and therefore drifts to $+\infty$. Therefore, $\kappa< \infty$ if $ \xi$ drifts to $-\infty$, and by definition of $\kappa$, the set $ \{ 1,2, \ldots,\kappa\}$ is accessible by $X_{n}$ for every $n \geq 1$. For irreducible Markov chains, one can evidently take $K=1$.
 
 A crucial consequence is that if \textbf{(A1)}, \textbf{(A2)}, \textbf{(A3)} hold and  the L\'evy process $ \xi$ drifts to $- \infty$, then $ \{ 1,2, \ldots,K\}$ is recurrent for the Markov chain, in the sense that for every $n \geq 1$, $ \inf \{ k \geq 1; X_{n}(k) \leq K\}<\infty$ almost surely (see Lemma \ref {lem:utile}). Loosely speaking, we call this the recurrent case.

Finally, for every $n \geq 1$, let $ {X}^{\dagger}_{n}$ be the Markov chain $X_{n}$ stopped at its first visit to $ \{ 1,2, \ldots,K\} $, that is ${X}^{\dagger}_{n}(\cdot)=X_{n}( \cdot \wedge A^{(K)}_{n})$,  where $A^{(K)}_{n}= \inf \{ k \geq 1; X_{n}(k) \leq K\}$,
 with again the usual convention  $ \inf \emptyset = \infty$.

\begin {thm}\label {thm:main2} Assume that  \textbf{(A1)}, \textbf{(A2)}, \textbf{(A3)} hold  and that  the L\'evy process $ \xi$ drifts to $- \infty$. Then the convergence
\begin{equation}
\label{eq:cvthm2} \left( \frac{ {X}^{\dagger} _{n}( \lfloor a_{n} t \rfloor)}{n} ; t \geq 0 \right)   \quad\mathop{\longrightarrow}^{(d)}_{n \rightarrow \infty} \quad (Y(t); t\geq 0)
\end{equation}
holds in distribution  in $ \D(\R_{+},\R)$.
\end {thm}

In this case, the process $Y$ is absorbed once it reaches $0$ (see the second and third images from the left in Fig.~\ref{fig:MC}). This result extends  \cite[Theorem 1]{HM11}, see Section \ref{sec:HM} for details. We will discuss in Section \ref {sec:nona} what happens when the Markov chain $ X_{n}$ is not stopped anymore. Observe that according to the asymptotic behavior of $ \xi$, the behavior of $Y$ is drastically different: when  $ \xi$ drifts to $- \infty$,  $Y$ is absorbed at $0$ at a finite time  and $Y$ remains forever positive otherwise. 

Let us mention that with the same techniques, it is possible to extend Theorems \ref {thm:main1} and  \ref {thm:main2} when the L\'evy process $ \xi$ is killed at a random exponential time, in which case $Y$  reaches $0$ by a jump. However, to simplify the exposition, we shall not pursue this goal here.

Given $ \sigma ^{2} \geq 0$, $ b \in \R $, $ \gamma > 0$ and a measure $ \Pi$ on $ \R \backslash \{ 0\}$ such that \eqref{eq:condPi} holds and $ \Pi ( \{ -1,1\} )=0$, it is possible to check the existence of a family $(p_{n,k}; n,k \geq 1)$ such that \textbf{(A1)} and \textbf{(A2)}, hold (see e.g.~\cite[Proposition 1]{HM11} in the non-increasing case). We may further request \textbf{(A3)} whenever $\int_1^{\infty} e^{\beta x} \Pi(dx)<\infty$ for some $\beta >0$. As a consequence, our Theorems \ref {thm:main1} and   \ref {thm:main2}  show that any nonnegative self-similar Markov process, such that its associated L\'evy measure $ \Pi$ has a small finite exponential moment on $[1,\infty)$, considered up to its first hitting time of the origin is the scaling limit of a Markov chain.

\subsection {Convergence of the absorption time} 

It is natural to ask whether the convergence \eqref{eq:cvthm2} holds jointly with the convergence of the associated absorption times.  Observe that this is not a mere consequence of Theorem \ref {thm:main2}, since absorption times, if they exist, are in general not continuous functionals for the Skorokhod topology on $ \D(\R_{+},\R)$.  Haas \& Miermont \cite[Theorem 2]{HM11} proved that, indeed, the associated absorption time converge for non-increasing Markov chains. We will prove that, under the same assumptions as for Theorem \ref {thm:main2},  the associated absorption times converge in distribution, and further the convergence holds also for the expected value under an additional  positive-recurrent type assumption.

 Let $ \Psi$ be the Laplace exponent associated with $\xi$, which is given by $$ \Psi(\lambda)= \Phi(-i \lambda)=\frac{1}{2} \sigma^{2} \lambda^{2} + b \lambda+ \int_{- \infty}^{\infty} \left( e^{\lambda x}-1 - \lambda x \mathbbm {1}_{|x| \leq 1} \right)  \ \Pi(dx).$$
for those values of $ \lambda \in \R$ such that this quantity is well defined, so that $ \Es { e^{\lambda \xi(t)}}=e^{ t\Psi(\lambda)}$. Note that $\textbf{(A3)}$ implies that $ \Psi$ is well defined on a positive neighborhood of $0$. 

\begin{framed}
\textbf{(A4)}. There exists $ \beta_{0}> \gamma$ such that
 \begin{equation}
 \label{eq:H4}\displaystyle \limsup_{n \rightarrow \infty} a_{n} \cdot \int_{1}^{\infty}  e^{\beta_{0} x }\ \Pi^{\ast}_{n}(dx) < \infty \qquad \textrm { and } \qquad  \Psi(\beta_{0})<0.
 \end{equation}
\end {framed}

Note the difference with \textbf{(A3)}, which only requires  the first inequality of \eqref{eq:H4} to hold for a certain $ \beta_{0}>0$. Also, if \textbf{(A4)} holds, then we have $ \Psi( \gamma)<0$ by convexity of $ \Psi$. Conversely, observe that \textbf{(A4)} is automatically satisfied if $ \Psi( \gamma)<0$ and the Markov chain has  uniformly bounded upwards jumps. 

 A crucial consequence is that if \textbf{(A1)}, \textbf{(A2)} and \textbf{(A4)} hold, then   the L\'evy process $ \xi$ drifts to $- \infty$ and   the first hitting time $ A_{n}^{(k)}$ of $ \{ 1,2, \ldots,k\}$ by $X_{n}$ has finite expectation for every  $n> k$, where $k$ is sufficiently large (see Lemma \ref {lem:utile2}). Loosely speaking, we call this the positive recurrent case. 
 
\begin {thm}\label{thm:absorption}Assume that \textbf{(A1)}, \textbf{(A2)}, \textbf{(A3)} hold and that $ \xi$ drifts to $- \infty$.  Let $K \geq 1$ be such that $ \{ 1,2, \ldots,K\}$ is accessible by $X_{n}$ for every $n \geq 1$.
\begin{enumerate}
\item[(i)] 
We have \begin{equation}
\label{eq:cvtemps} \frac{A_{n}^{(K)}}{a_{n}}\quad\mathop{\longrightarrow}^{(d)}_{n \rightarrow \infty} \quad \int_{0}^{\infty} e^{\gamma \xi(s)}ds,
\end{equation}
and this convergence holds jointly with \eqref{eq:cvthm2}. 
\item[(ii)] If further \textbf{(A4)} holds, and in addition, 
 \begin{equation}
 \label{eq:H42} \textrm{for every } n \geq K+1, \qquad \sum_{k \geq 1} k^{\beta_{0}} \cdot p_{n,k}< \infty,
 \end{equation}
then
\begin{equation}
\label{eq:CVL1} {\frac{\Es{A^{(K)}_{n}}}{a_{n}}} \quad\mathop{\longrightarrow}_{n \rightarrow \infty} \quad  \frac{1}{|\Psi(\gamma)|}.
\end{equation}
\end{enumerate}
\end {thm}

We point out that when \eqref{eq:H4} is satisfied, the inequality $ \sum_{k \geq 1} k^{\beta_{0}} \cdot p_{n,k}< \infty$ is automatically satisfied for every $n$ sufficiently large, that is  condition \eqref{eq:H42} is then fulfilled provided that $K$ has been chosen sufficiently large. See Remark \ref {rem:moments} for the extension of \eqref{eq:CVL1} to higher order moments.  Finally, observe that \eqref{eq:H42} is the only condition which does not only depend on the asymptotic behavior of $p_{n, \cdot}$ as $ n \rightarrow \infty$ (the behavior of the law of $X_{n}(1)$ for small values of $n$ matters here).

This result has been proved by Haas \& Miermont  \cite[Theorem 2]{HM11} in the case of non-increasing Markov chains. However, some differences appear in our more general setup. For instance, \eqref{eq:CVL1} always is true when the chain is non-increasing, but clearly cannot hold if $ \Psi(\gamma)>0$ (in this case $\int_{0}^{\infty} e^{\gamma \xi(s)}ds=\infty$ a.s.) or if the Markov chain is irreducible and not positive recurrent (in this case $\E{A^{(K)}_{n}}=\infty$).

\subsection {Scaling limits for the non-absorbed Markov chain}
\label {sec:nona}

It is natural to ask if Theorem \ref {thm:main2} also holds for the non-absorbed Markov chain $X_{n}$. Roughly speaking, we show that the answer is affirmative if it does not make too large jumps when reaching low values belonging to $ \{ 1,2, \ldots,K\}$, as quantified by the following last assumption which completes \textbf{(A4)}.

\begin{framed}
\textbf{(A5)}. Assumption \textbf{(A4)} holds and, in addition, for every $n\geq 1$, we have
$$
 \Es{X_n(1)^{\beta_0}}= \sum_{k \geq 1} k^{\beta_{0}} \cdot p_{n,k} \ < \ \infty,
$$
with $\beta_{0}>\gamma$ such that \eqref{eq:H4} holds. 
\end {framed}

\begin {thm}\label {thm:main3} Assume that  \textbf{(A1)}, \textbf{(A2)} and \textbf{(A5)} hold.  Then the convergence
\begin{equation}
\label{eq:cvthm3} \left( \frac{ {X} _{n}( \lfloor a_{n} t \rfloor)}{n} ; t \geq 0 \right)   \quad\mathop{\longrightarrow}^{(d)}_{n \rightarrow \infty} \quad (Y(t); t\geq 0)
\end{equation}
holds in distribution  in $ \D(\R_{+},\R)$.
\end {thm}
 Recall that when $ \xi$ drifts to $-\infty$, we have $I_{\infty}< \infty$ and $Y_{t}=0$ for $t \geq I_{\infty}$, so that roughly speaking this result tells us that with probability tending to $1$ as $n \rightarrow \infty$, once $X_{n}$ has reached levels of order $o(n)$, it will remain there on time scales of order $a_{n}$.

If \textbf{(A4)} holds but not \textbf{(A5)}, we believe that the result of Theorem \ref {thm:main3} does not hold in general since the Markov chain may become null-recurrent (see Remark \ref {rem:inf}) and the process may ``restart'' from $0$ (see Section \ref {sec:questions}).

\subsection {Techniques}
We finally briefly comment on the techniques involved in the proofs of Theorems \ref {thm:main1} and \ref {thm:main2},  which differ from those of  \cite {HM11}. We start by embedding $X_{n}$ in continuous time by considering an independent Poisson process $ \mathcal{N}_{n}$ of parameter $a_{n}$, which allows us to construct a continuous-time Markov process $L_{n}$ such that the following equality in distribution holds
$$ \left(  \frac{1}{n} X_{n}(  { \mathcal{N}_{n}(t) }) ; t \geq 0 \right)  \quad\mathop{=}^{(d)}  \quad \left( \exp(L_{n}(\tau_{n}(t))); t \geq 0 \right),$$ 
where $ \tau_{n}$ is a Lamperti-type time change of $L_{n}$ (see~\eqref{eq:timechange}). Roughly speaking, to establish Theorems \ref {thm:main1} and \ref {thm:main2}, we use the characterization of functional convergence of Feller processes by generators in order to show that $L_{n}$ converges in distribution to $ \xi$ and that $\tau_{n}$ converges in distribution towards $ \tau$. However, one needs to proceed with particular caution when $ \xi$ drifts to $ - \infty$, since  the time changes then explode. In this case, assumption \textbf{(A3)} will give us useful bounds on the growth of $X_{n}$ by Foster--Lyapounov techniques.

\section {An auxiliary continuous-time Markov process}
\label {sec:aux}

In this section, we construct an auxiliary continuous-time Markov chain $(L_{n}(t); t \geq 0)$ in such a way that $L_{n}$, appropriately scaled, converges to $ \xi$ and such that, roughly speaking, ${ X}_{n}$ may be recovered from $ \exp(L_{n})$ by a Lamperti-type time change.

\subsection {An auxiliary continuous-time Markov chain $L_{n}$}
\label {sec:compound}

For every $ n \geq 1$, first let $ (\xi_{n}(t); t \geq 0)$ be a compound Poisson process with L\'evy measure $a_{n} \cdot \Pi_{n}^{\ast}$. That is
$$ \Es {e^{ i \lambda \xi_{n}(t)}}= \exp \left( t \int_{- \infty}^{ \infty}(e^{i \lambda x}-1) \cdot   a_{n} \Pi^{\ast}_{n}(dx) \right), \qquad \lambda \in \R, t \geq 0.$$
It is well-known that $ \xi_{n}$ is a Feller process on $ \R$ with generator $ \mathcal{A}_{n}$ given by
$$ \mathcal{A}_{n}f(x)= a_{n} \int_{- \infty}^{\infty} \left( f(x+y)-f(x) \right)  \ \Pi_{n}^{\ast}(dy), \qquad f \in \mathcal{C}^{\infty}_{c}( \R), \quad x \in \R,$$  
where $\mathcal{C}^{\infty}_{c}( I)$ denotes the space of real-valued infinitely differentiable functions with compact support in an interval $I$.

It is also well-known that the L\'evy process $ \xi$, which has been introduced in Section \ref {sec:descr}, is  a Feller process on $ \R_{}$ with infinitesimal generator $ \mathcal{A}$ given by
$$  \mathcal{A}f(x)= \frac{1}{2} \sigma^{2} f''(x)+bf'(x)+ \int_{- \infty}^{\infty} \left( f(x+y)-f(x)-f'(x) y \mathbbm {1}_{|y| \leq 1} \right)  \ \Pi(dy), \qquad f \in \mathcal{C}^{\infty}_{c}( \R), \quad x \in 	\R,$$
and, in addition, $\mathcal{C}^{\infty}_{c}( \R)$ is a core for $ \xi$ (see e.g.~\cite[Theorem 31.5]{Sat13}). Under \textbf{(A1)} and \textbf{(A2)}, by \cite[Theorems 15.14 \& 15.17]{Kal02},  $\xi_{n}$ converges in distribution in $ \D(\R_{+},\R)$ as $ n \rightarrow \infty$ to $ \xi$. It is then classical that the  convergence of generators
\begin{equation}
\label{eq:cvgen} \mathcal{A}_{n}f  \quad\mathop{\longrightarrow}_{n \rightarrow \infty} \quad  \mathcal{A}f
\end{equation}
holds for every $f \in \mathcal{C}^{\infty}_{c}( \R)$, in the sense of the uniform norm on $\mathcal{C}_{0}( \R)$. It is also possible to check directly \eqref{eq:cvgen}  by a simple calculation which relies on the fact that $ \lim_{ \epsilon \rightarrow 0} \lim_{n \rightarrow \infty} a_{n} \int_{- \epsilon}^{\epsilon} y^{3}\  \Pi_{n}^{\ast}(dy)=0 $ by \textbf{(A2)} (see Sec.~\ref {sec:zd} for similar estimates). We leave the details to the reader.

For $ x \in \R$, we let $ \{ x\}=x- \fl {x}$ denote the fractional part of $x$ and also set $ \ce {x}= \fl {x}+1$  (in particular  $\ce {n}=n+1$ if $n$ is an integer). By convention, we set $ \mathcal{A}_{0}=0$ and $ \Pi^{\ast}_{0}=0$.  Now introduce an auxiliary continuous-time Markov chain $(L_{n}(t); t \geq 0)$ on $ \R \cup \{ + \infty\} $ which has generator $ \mathcal{B}_{n}$ defined as follows:
\begin{equation}
\label{eq:genLn}\mathcal{B}_{n}f(x)= \left( 1- \{ ne^{x}\} \right) \cdot \mathcal{A}_{ \fl {ne^{x}}}f(x)+    \{ ne^{x}\} \cdot \mathcal{A}_{ \ce {ne^{x}}}f(x)), \qquad f \in \mathcal{C}^{\infty}_{c}( \R), \quad x \in \R.
\end{equation}
We allow $L_{n}$ to take eventually the cemetery value $+ \infty$, since it is not clear for the moment whether $L_{n}$ explodes in finite time or not. The process $L_{n}$ is designed in such a way that if $n \exp(L_{n})$ is at an integer valued state, say $j \in \N$, then it will wait a random time distributed as an exponential random variable of parameter $a_{j}$ and then jump to state $k \in \N$ with probability $p_{j,k}$ for $k \geq 1$. In particular $n \exp(L_{n})$
then remains integer whenever it starts in $\N$. 
Roughly speaking, the generator \eqref{eq:genLn} then extends the possible states of  $L_{n}$ from $  \ln(\N/n)$ to $\R$ by smooth interpolation.  

A crucial feature of $L_{n}$ lies in the following result.

\begin {prop}\label {prop:cvL} Assume that  \textbf{(A1)} and \textbf{(A2)} hold. For every $x \in \R$,  $L_{n}$, started from  $x$, converges in distribution in $ \D(\R_{+},\R)$ as $n \rightarrow \infty$ to $ \xi+x$.
\end {prop}

\begin {proof} Consider the modified continuous-time Markov chain $(\widehat{L}_{n}(t); t \geq 0)$ on $ \R$ which has generator $ \widehat{\mathcal{B}}_{n}$ defined as follows:
\begin{equation}
\label{eq:genLntilde}\widehat{\mathcal{B}}_{n}f(x)= \left( 1- \{ ne^{x}\} \right) \cdot \mathbbm {1}_{ \fl {ne^{x}} \leq n^{2}} \cdot \mathcal{A}_{ \fl {ne^{x}}}f(x)+    \{ ne^{x}\}\cdot \mathbbm {1}_{ \ce {ne^{x}} \leq n^{2}}  \cdot\mathcal{A}_{ \ce {ne^{x}}}  f(x), \quad  f \in \mathcal{C}^{\infty}_{c}( \R), \ x \in \R.
\end{equation}
We stress that $\widehat{\mathcal{B}}_{n}f(x)= {\mathcal{B}}_{n}f(x)$ for all $x < \ln n$, so the processes $L_{n}$ and $\widehat{L}_{n}$ can be coupled so that their trajectories coincide up to the time when they exceed $\ln n$. Therefore, it is  enough to check that for every $x \geq 0$,  $\widehat{L}_{n}$, started from $x$, converges in distribution in $ \D(\R_{+},\R)$ to $ \xi + x$.  

The reason for introducing $ \widehat {L}_{n}$ is that clearly $ \widehat{ L}_{n}$ does not explode, and is in addition a Feller process (note that it is not clear \emph{a priori}  that $L_{n}$ is a Feller process that does not explode).  Indeed, the generator $ \widehat{\mathcal{B}}_{n}$ can be written in the form $ \widehat{\mathcal{B}}_{n}f(x)= \int_{- \infty}^{\infty}(f(x+y)-f(x)) \mu_{n}(x,dy)$ for $x \in \R$ and $f \in \mathcal{C}_{c}^{\infty}(\R)$ and where $\mu_{n}(x,dy)$ is the measure on $ \R$ defined by
$$ \mu_{n}(x,dy)=\left( 1- \{ ne^{x}\} \right) \mathbbm {1}_{ \fl {ne^{x}} \leq n^{2}} {a}_{ \fl {ne^{x}}}  {\Pi^{\ast}}_{ \fl {ne^{x}}}(dy)+    \{ ne^{x}\} \mathbbm {1}_{ \ce {ne^{x}} \leq n^{2}} a_{ \ce {ne^{x}}} \Pi^{\ast}_{ \ce {ne^{x}}}(dy).$$
It is straightforward to  check that $ \sup_{x \in \R} \mu_{n}(x, \R)< \infty$ and that the map $x \rightarrow \mu_{n}(x,dy)$ is weakly continuous. This implies that $\widehat{L}_{n}$ is indeed a Feller process.

By \cite[Theorem 19.25]{Kal02} (see also Theorem 6.1 in \cite[Chapter 1]{EK86}), in order to establish Proposition \ref {prop:cvL} with $L_n$ replaced by $\widehat {L}_n$, it is enough to check that $ \widehat{\mathcal{B}}_{n}f$ converges uniformly to $ \mathcal{A}f$ as $n \rightarrow \infty$ for every  $f \in \mathcal{C}_{c}^{\infty}(\R)$. For the sake of simplicity, we shall further suppose that $|f|\leq 1$. Note that $ \mathcal{A} f(x) \rightarrow0$ as $x \rightarrow  \pm \infty$ since $ \xi$ is a Feller process, and \eqref{eq:cvgen} implies that     $ \widehat{\mathcal{B}}_{n}f$ converges uniformly on compact intervals to $ \mathcal{A}f$ as $n \rightarrow \infty$. Therefore, it is enough to check that
$$ \lim_{M \rightarrow \infty} \lim_{n \rightarrow \infty} \sup_{|x|>M} |\widehat{\mathcal{B}}_{n}f(x)|=0.$$
To this end, fix $ \epsilon>0$. By \eqref{eq:condPi}, we may choose $u_0>0$ such that $ \Pi(\R \backslash (-u_0,u_0)) < \epsilon$. The portmanteau theorem  \cite[Theorem 2.1]{Bil99} and \textbf{(A1)} imply that 
$$ \limsup_{n \rightarrow \infty} a_{n} \cdot \Pi^{\ast}_{n}( \R \backslash (-u_0,u_0))  \leq  \Pi(\R \backslash (-u_0,u_0))<\epsilon.$$
We can therefore find $M>0$ such that $a_{n} \cdot \Pi^{\ast}_{n}( \R \backslash (-M,M))  < \epsilon$ for every $n \geq 1$. Now let $m_{0}<M_{0}$ be such that the support of $f$ is included in $[m_{0},M_{0}]$. Then, for $x>M>M_{0}+u_{0}$,
$$ \widehat{\mathcal{B}}_{n}f(x)= \int_{- \infty}^{\infty}f(x+y) \mathbbm {1}_{x+y \leq M_{0}} \ \mu_{n}(x,dy),$$
so that $ |\widehat{\mathcal{B}}_{n}f(x)| \leq  {a}_{ \fl {ne^{x}}}  {\Pi^{\ast}}_{ \fl {ne^{x}}}((- \infty,M_{0}-M))+  {a}_{ \ce {ne^{x}}}  {\Pi^{\ast}}_{ \ce {ne^{x}}}((-\infty,M_{0}-M)) \leq  2 \epsilon$.  One similarly shows that  $ |\widehat{\mathcal{B}}_{n}f(x)| \leq  2 \epsilon$ for $x<-M<m_{0}-u_{0}$. This completes the proof.
\end {proof}

\subsection{Recovering $X_{n}$ from $L_n$ by a time change}
Unless  otherwise specifically  mentioned, we shall henceforth assume that $L_n$ starts from $0$. 
In order to formulate a connection between $X_{n}$ and $\exp(L_{n})$, it is convenient to introduce some additional randomness. Consider a Poisson process $ ( \mathcal{N}_{n}(t); t \geq 0)$ of intensity $a_{n}$ independent of $X_{n}$, and, for every $t \geq 0$, set
\begin{equation}
\label{eq:timechange}\tau_{n}(t)= \inf \left\{ u \geq  0 ; \int_{0}^{u} \frac{a_{n \exp(L_{n}(s))}}{a_{n}} \ ds >t \right\}.
\end{equation}
We stress that $\tau_n(t)$ is finite a.s.~for all $t\geq 0$. Indeed, if we write $\zeta$ for the possible explosion time of $L_n$ ($\zeta=\infty$ when $L_n$ does not explode), then
$\int_{0}^{\zeta} a_{n \exp(L_{n}(s))}ds=\infty$ almost surely. Specifically, when $n \exp(L_{n})$ is at some state, say $k$, it stays there for an exponential time with parameter $a_k$ and 
the contribution of this portion of time to the integral has thus the standard exponential distribution, which entails our claim. 

\begin{lem}\label {lem:coupling}Assume that $L_{n}(0)=0$. Then we have
\begin{equation}
\label{eq:MC} \left(  \frac{1}{n} X_{n}(  { \mathcal{N}_{n}(t) }) ; t \geq 0 \right)  \quad\mathop{=}^{(d)}  \quad \left( \exp(L_{n}(\tau_{n}(t))); t \geq 0 \right).
\end{equation}
\end {lem} 

\begin {proof} Plainly, the two processes appearing in  \eqref{eq:MC} are continuous-time Markov chains, so to prove the statement, we need to check that their respective embedded discrete-time Markov chains (i.e.~jump chains) have the same law, and that the two exponential waiting times at a same state have the same parameter. 

Recall the description made after \eqref{eq:genLn} of the process $n \exp(L_{n})$ started at an integer value. We see in particular that the two jump chains  in \eqref{eq:MC} have indeed the same law. Then fix some $j\in\N$ and recall that the waiting time of $L_{n}$ at state $\ln (j/n)$ is distributed according to an exponential random variable of parameter $a_{j}$. It follows readily from the definition of the time-change $\tau_n$ that the waiting time of $\exp(L_{n}(\tau_n(\cdot))$ at state $j/n$ is distributed according to an exponential random variable of parameter $a_{j}\times \frac{a_n}{a_j}= a_n$. 
This proves our claim.
\end {proof}

\section {Scaling limits of the Markov chain $X_{n}$}
\label {sec:scaling}

\subsection{The non-absorbed case: proof of Theorem \ref {thm:main1}}
\label {sec:nonabsorbed}

We now prove  Theorem \ref {thm:main1} by establishing that
\begin{equation}
\label{eq:case1} \left( \frac{ {X} _{n}( \mathcal{N} _{n} (t) )}{n} ; t \geq 0 \right)   \quad\mathop{\longrightarrow}^{(d)}_{n \rightarrow \infty} \quad (Y(t); t\geq 0)
\end{equation}
in $\D(\R_{+},\R)$. Since by the functional law of large numbers $(\mathcal{N} _{n}(t)/ a_{n}; t \geq 0)$ converges in probability to the identity uniformly on compact sets, Theorem \ref {thm:main1} will follow from \eqref{eq:case1} by standard properties of the Skorokhod topology (see e.g.~\cite[VI. Theorem 1.14]{JS03}).

\begin {proof}[Proof of Theorem \ref {thm:main1}]
Assume that \textbf{(A1)}, \textbf{(A2)} hold and that $ \xi$ does not drift to $- \infty$. In particular, recall from the Introduction that  we have $I_{ \infty}= \infty$ and the process  $Y(t)= \exp ( \xi(\tau(t)))$ remains bounded away from $0$ for all $t \geq 0$. 

By standard properties of regularly varying functions (see e.g.~\cite[Theorem 1.5.2]{BGT87}), $x \mapsto a_{\fl{nx}}/a_{n}$ converges uniformly on compact subsets of $ \R_{+}$ to $x \mapsto x^{\gamma}$ as $n \rightarrow \infty$. Recall that $L_{n}(0)=0$. Then by Proposition \ref {prop:cvL} and standard properties of the Skorokhod topology (see e.g.~\cite[VI. Theorem 1.14]{JS03}), it follows that $$ \left( \frac{a_{n \exp(L_{n}(s))}}{a_{n}}; s \geq 0 \right)   \quad\mathop{\longrightarrow}^{(d)}_{n \rightarrow \infty} \quad \left(  \exp ( \gamma \xi(s)); s \geq 0 \right)$$
in $\D(\R_{+},\R)$. This implies that
\begin{equation}
\label{eq:int}\left( \int_{0}^{u}  \frac{a_{n \exp(L_{n}(s))}}{a_{n}} ds ; u \geq 0 \right)   \quad\mathop{\longrightarrow}^{(d)}_{n \rightarrow \infty} \quad   \left(    \int_{0}^{u}   \exp(\gamma \xi(s)) ds ; u \geq 0 \right),
\end{equation}
in $ \mathcal{C}(\R_{+},\R)$, which is the space of real-valued continuous functions on $ \R_{+}$ equipped with the topology of uniform convergence on compact sets. Since the two processes appearing in \eqref{eq:int}  are almost surely (strictly) increasing in $u$ and  $I_{\infty}=\infty$, $ \tau$ is almost surely (strictly) increasing and continuous on $ \R_{+}$. It is then a simple matter to see that \eqref{eq:int}  in turn implies that $\displaystyle \tau_{n}$ converges in distribution to $\tau$ in $ \mathcal{C}(\R_{+},\R)$.  Therefore, by applying Proposition \ref {prop:cvL} once again, we finally get that 
$$\left( \exp(L_{n}(\tau_{n}(t))); t \geq 0 \right)  \quad\mathop{\longrightarrow}^{(d)}_{n \rightarrow \infty} \quad  \left(  \exp ( \xi( \tau(t))) ; t \geq 0\right)=Y$$
in $\D(\R_{+},\R)$. 
 By Lemma \ref{lem:coupling}, this establishes \eqref{eq:case1} and completes the proof.
\end {proof}

\subsection {Foster--Lyapounov type estimates}
\label {sec:FL}

Before tackling the proof of Theorem \ref {thm:main2}, we start by exploring several preliminary consequences of \textbf{(A3)}, which will also be useful in Section \ref {sec:abs}.

In the  irreducible case,  Foster \cite{Fos53} showed that the Markov chain $X$ is positive recurrent if and only if there exists a finite set $S_{0} \subset \N$, a function $f: \N \rightarrow \R_{+}$ and $ \epsilon>0$ such that
\begin{equation}
\label{eq:FL} \textrm {for every } i \in S_{0}, \quad \sum_{j \geq 1} p_{i,j} f(j)< \infty, \qquad \textrm {and} \qquad \textrm {for every } i \not \in S_{0}, \quad \sum_{j \geq 1} p_{i,j} f(j) \leq  f(i)-\epsilon.
\end{equation}
The map $f: \N \rightarrow \R_{+}$ is commonly referred to as a Foster--Lyapounov function. The conditions \eqref{eq:FL} may be rewritten in the equivalent forms
$$\textrm {for every } i \in S_{0},  \quad \Es {f(X_{i}(1))}< \infty, \qquad \textrm {and} \qquad \textrm {for every } i \not \in S_{0}, \quad \Es {f(X_{i}(1))-f(i)} \leq  -\epsilon.$$
Therefore, Foster--Lyapounov functions allow to construct nonnegative supermartingales, and the criterion may be interpreted as a stochastic drift condition in analogy with Lyapounov's stability criteria for ordinary differential equations. A similar criterion exists for recurrence instead of positive recurrence (see e.g.~\cite[Chapter 5]{Bre99} and \cite {MT09}).

In our setting, we shall see that \textbf{(A3)} yields Foster--Lyapounov functions of the form $f(x)=x^{\beta}$ for certain values of $ \beta>0$. For $i,K \geq 1$, recall that  $A^{(K)}_{i}= \inf\{j \geq 1; X_i(j) \leq K\}$ denotes the first return time of  $X_{i}$ to $ \{ 1,2, \ldots,K\} $. 
 
\begin {lem}\label{lem:utile} Assume that \textbf{(A1)}, \textbf{(A2)}, \textbf{(A3)} hold and that the L\'evy process $ \xi$ drifts to $- \infty$. Then:
\begin{enumerate}
\item[(i)] There exists $ 0<\beta_{0} <\beta$ such that $ \Psi(\beta_{0})<0$.
\item[(ii)] For all such $\beta_0$, we have 
\begin{equation}
\label{eq:claim} a_n \int_{0}^{\infty} \left( e^{\beta_{0}x}-1 \right)  \Pi_n^{\ast}(dx)  \quad\mathop{\longrightarrow}_{n \rightarrow \infty} \quad  \Psi(\beta_{0})<0.
\end{equation}
\item[(iii)]  Let $M \geq K$ be such that $a_n \int_{0}^{\infty} \left( e^{\beta_{0}x}-1 \right)  \Pi_n^{\ast}(dx)  \leq 0$ for every $n \geq M$. Then, for every $i \geq M$, the process defined by $ \mathcal{M}_{i}(\cdot)=X_{i}(\cdot \wedge A^{(M)}_{i})^{\beta_{0}}$ is a positive supermartingale (for the  canonical filtration of $X_{i}$). 
\item[(iv)] Almost surely, $ A^{(K)}_{i}< \infty$ for every $i \geq 1$.
\end{enumerate}
\end {lem}

\begin {proof}
 By \textbf{(A1)} and \textbf{(A3)}, we have $ \int_{1}^{\infty} x \ \Pi(dx)< \infty$. Since $ \xi$ drifts to $- \infty$, by \cite[Theorem 1]{BY05}, we have    
$$b+ \int_{|x|>1} x \Pi(dx)\in[-\infty, 0).$$
In particular, $ \Psi'(0+)=b+ \int_{|x|>1} x \Pi(dx)\in[-\infty, 0)$, so that there exists $ \beta_{0}>0$ such that $ \Psi( \beta_{0})<0$. This proves (i).

For the second assertion, recall from Section~\ref {sec:compound} that $ \xi_n$ is  a compound Poisson Process with L\'evy measure $a_n \cdot \Pi_n^{\ast}$ that converges in distribution to $ \xi$ as $n \rightarrow \infty$. By dominated convergence, this implies that $ \E {e^{\beta_0 \xi_n(1)}} \rightarrow  \E {e^{\beta_0 \xi(1)}}$ as $n \rightarrow \infty$, or, equivalently, that \eqref{eq:claim} holds. 

For (iii), note that for $i \geq M$,
\begin{equation}
\label{eq:fos} \frac{a_{i}}{i^{\beta_{0}}} \cdot \Es {X_{i}(1)^{\beta_{0}}-X_{i}(0)^{\beta_{0}}}=   \frac{a_{i}}{i^{\beta_{0}}} \cdot \ \sum_{k=1} ^{\infty} p_{i,k} \left( k^{\beta_{0}}-i^{\beta_{0}} \right) =a_{i}\cdot \int_{0}^{\infty} \left( e^{\beta_{0}x}-1 \right)  \Pi_{i}^{\ast}(dx)\leq 0.
\end{equation}
Hence  $\Es {X_i(1)^{\beta_{0}}} \leq  \Es{X_i(0)^{\beta_{0}}}$ for every $i \geq M$, which implies that $ \mathcal{M}_{i}$ is a positive supermartingale.

The last assertion is an analog of Foster's criterion of recurrence for irreducible Markov chains. Even though we do not assume irreducibility here, it is a simple matter to adapt the proof of Theorem 3.5 in \cite[Chapter 5]{Bre99} in our case. Since $ \mathcal{M}_{i}$ is a positive supermartingale, it converges almost surely to a finite limit, which implies that $ A^{(M)}_{i}< \infty$ almost surely for every $i \geq M+1$, and therefore $ A^{(M)}_{i}< \infty$  for every $i \geq 1$ (by an application of the Markov property at time $1$). Since $ \{ 1,2, \ldots,K\} $ is accessible by $X_{n}$ for every $n \geq 1$,  it readily follows  that $A^{(K)}_{i}< \infty$ almost surely for every $i \geq K+1$.  
\end {proof}

We point out that the recurrence of the discrete-time chain $X_n$ entails that the continuous-time process $L_n$ defined in Section \ref {sec:compound} does not explode (and, as a matter of fact, is also recurrent). If the stronger assumptions \textbf{(A4)} and \eqref{eq:H42} hold instead of \textbf{(A3)}, roughly speaking the Markov chain becomes positive recurrent (note that $\xi$ drifts to $-\infty$ when \textbf{(A4)} holds):

\begin {lem}\label{lem:utile2}  Assume that \textbf{(A1)}, \textbf{(A2)},  \textbf{(A4)} and  \eqref{eq:H42} hold. Then:
\begin{enumerate}
\item[(i)] There exists  an integer $M \geq K$ and a constant $c>0$ such that, for every $n \geq M$,
\begin{equation}
\label{eq:recpos}a_n \cdot \int_{-\infty}^{\infty} \left( e^{\beta_{0}x}-1 \right)  \Pi_n^{\ast}(dx) \quad \leq \quad -c.
\end{equation}
\item[(ii)] For every $n\geq K+1$, $ \Es{A^{(K)}_{n}}< \infty$.
\item[(iii)] Assume  that, in addition, \textbf{(A5)} holds. Then  for every $n \geq 1$, $ \Es{A^{(K)}_{n}}< \infty$ .
\end{enumerate}
\end {lem}

\begin {proof} The proof of (i) is similar to that of Lemma \ref{lem:utile}. 
For the other assertions, it is convenient to consider the following modification of the Markov chain. We introduce probability transitions $p'_{n,k}$ such that $p'_{n,k}= p_{n,k}$ for all $k\geq 1$ and $n>K$,
and for $n=1, \ldots, K$, we  choose the $p'_{n,k}$ such that $p'_{n,k}>0$ for all $k\geq 1$ and $\sum_{k \geq 1} k^{\beta_{0}} \cdot p'_{n,k}< \infty$. In other words, the modified chain with transition probabilities $p'_{n,k}$, say $X'_n$, then fulfills \textbf{(A5)}. 

The chain  $X'_n$  is then irreducible (recall that, by assumption, $\{1,\ldots, K\}$ is accessible by $X_n$ for every $n \in \N$) and fulfills the assumptions of Foster's Theorem. See 
e.g.~Theorem 1.1 in Chapter 5 of \cite{Bre99} applied with $h(i)=i^{\beta_{0}}$ and $F=\{1,\ldots , M\}$. Hence  $X'_n$ is positive recurrent, and as a consequence,  
the first entrance time of $X'_n$ in $\{1,\ldots, K\}$ has finite expectation for every $n \in \N$. But by construction, for every $n\geq K+1$, the chains $X_n$ and $X'_n$ coincide until the first entrance in $\{1,\ldots, K\}$; this proves (ii). Finally, when \textbf{(A5)} holds, there is no need to modify $X_n$ and the preceding argument shows that $ \Es{A^{(K)}_{n}}< \infty$ for all $n\geq 1$. 
 \end {proof}
 
 \begin {rem}We will later check that under the assumptions of Lemma \ref {lem:utile2}, we may have  $ \E{A^{(K)}_{i}}= \infty$ for some $1 \leq i \leq K$  if  \textbf{(A4)} holds but not \textbf{(A5)} (see Remark \ref {rem:inf}).
\end {rem}

Recall that $L_n$ denotes the auxiliary continuous-time Markov chain which has been defined in Section \ref {sec:compound} with $L_n(0)=0$.

\begin{cor} \label{cor:super} Keep the same assumptions and notation as in Lemma \ref{lem:utile2}, and introduce the first passage time
$$\alpha^{(M)}_n=\inf\{t\geq 0; \ n \exp(L_n(t))\leq M\}.$$
The process
$$\exp\left( \beta_0 L_n(t\wedge \alpha_n^{(M)}) + c (t\wedge \alpha^{(M)}_n)\right),\qquad t\geq 0,$$
is then a supermartingale.
\end{cor}

\begin{proof}
Let $R>M$ be arbitrarily large; we shall prove our assertion with $\alpha^{(M)}_n$ replaced by 
$$\alpha^{(M,R)}_n=\inf\{t\geq 0 ; n \exp(L_n(t))\not\in \{ M +1, \ldots, R\} \}.$$
The process $L_n$ stopped at time $\alpha^{(M,R)}_n$ is a Feller process with values in $ - \ln n+ \ln \N$ , and it follows from \eqref{eq:genLn} that its infinitesimal generator, say ${\mathcal G}$, is given by
$${\mathcal G}f(x)=a_{ne^x}\int_{-\infty}^{\infty}\left( f(x+y)-f(x) \right) \Pi^\ast_{ne^x}(dy)$$
for every $x$ such that $ne^x\in \{ M +1, \ldots, R\}$. 
Applying this for $f(y)=\exp(\beta_0 y)$, we get from Lemma \ref{lem:utile2} (i) that ${\mathcal G}f(x)\leq -cf(x)$, which entails that 
$$f\left(   L_n(t\wedge \alpha_n^{(M,R)}) \right) \exp\left( c (t\wedge \alpha^{(M,R)}_n)\right),\qquad t\geq 0$$
is indeed a supermartingale. To conclude the proof, it suffices to let $R\to \infty$, recall that $L_n$ does not explode,  and apply the (conditional) Fatou Lemma.
\end{proof}

We  now establish two useful lemmas based on the Foster--Lyapounov estimates of Lemma \ref {lem:utile}. The first one is classical and states that if the L\'evy process $\xi$ drifts to $-\infty$ and its L\'evy measure $\Pi$ has finite exponential moments, then its overall supremum has an exponentially small tail. The second, which is the discrete counterpart of the first, states that if the Markov chain $X_{i}$ starts from a low value $i$, then $X_{i}$ will unlikely  reach a high value without entering $ \{ 1,2, \ldots,K\} $ first. 

\begin {lem}\label {lem:continuous} Assume that the L\'evy process $ \xi$ drifts to $- \infty$ and that its L\'evy measure fulfills the integrability condition $\int_1^{\infty}e^{\beta x} \ \Pi(dx)<\infty$ for some $\beta >0$.  There exists $\beta_0>0$ sufficiently small with $\Psi(\beta_0)<0$, and then 
for every $u\geq 0$, we have
$$ \P \left(  \sup_{s \geq 0}  \xi(s)>u \right)  \leq  e^{ -\beta_{0}u}.$$
\end {lem}

\begin {proof} The assumption on the L\'evy measure ensures that the Laplace exponent $\Psi$ of $\xi$ is well-defined and finite on $[0,\beta]$. Because $\xi$ drifts to $-\infty$, the right-derivative $\Psi'(0+)$ of the convex function $\Psi$ must be strictly negative (possibly $\Psi'(0+)=-\infty$) and therefore we can find $\beta_0>0$ with $\Psi(\beta_0)<0$.
Then the process $(e^{ \beta_{0} \xi(s)}, s \geq 0)$ is a nonnegative supermartingale and our claim follows from the optional stopping theorem applied at the first passage time above level $u$. \end {proof}

We now prove an analogous statement for the discrete Markov chain $X_{n}$, tailored for future use:

\begin {lem} \label {lem:discrete}Assume that \textbf{(A1)}, \textbf{(A2)}, \textbf{(A3)} hold and that  the L\'evy process $ \xi$ drifts to $- \infty$. Fix $ \epsilon>0$. For every $n$ sufficiently large, for every $1 \leq i \leq  \epsilon^{2} n$, we have $$ \Pr{ X_{i} \textrm{ reaches } [\epsilon n, \infty) \textrm { before }  [1,K]} \leq 2\epsilon^{\beta_{0}}.$$
\end {lem}

\begin {proof}We first check that there exists an integer $M \geq K$, such that for every $ 1 \leq  i \leq N$,
\begin{equation}
\label{eq:discrete1}\Pr{ X_{i} \textrm{ reaches } [N, \infty) \textrm { before }  [1,M]} \leq (i/N)^{\beta_{0}}.
\end{equation}
By Lemma \ref {lem:utile}, there exists $M \geq K$ such that $ \mathcal{M}_{i}(\cdot)=X_{i}^{\beta_{0}}(\cdot \wedge A^{(M)}_{i})$ is a positive supermartingale. Hence, setting $B^{(N)}_{i}=\inf\{j \geq 0; X_{i}(j) \geq  N\}$, by the optional stopping theorem we get that
$$i^{\beta_{0}} \geq  \Es {X_{i}^{\beta_{0}}(A^{(M)}_{i} \wedge B^{(N)}_{i})} \geq  \Es {X_{i}^{\beta_{0}}( B^{(N)}_{i}) \mathbbm {1}_{ \{ A^{(M)}_{i}> B^{(N)}_{i}\} }} \geq N^{\beta_{0}} \Pr {A^{(M)}_{i}>B^{(N)}_{i}}.$$
This establishes \eqref{eq:discrete1}.

We now turn to the proof of the main statement. By the Markov property, write
\begin{eqnarray*}
&& \Pr{ X_{i} \textrm{ reaches } [\epsilon n, \infty) \textrm { before }  [1,K]} \\
&& \qquad\qquad\qquad \leq   \Pr{ X_{i} \textrm{ reaches } [\epsilon n, \infty) \textrm { before }  [1,M]} + \sum_{j=K+1}^{M}  \Pr{X_{j} \textrm{ reaches } [\epsilon n, \infty) \textrm { before }  [1,K]}.
\end{eqnarray*}
By \eqref{eq:discrete1}, the first term of the latter sum is bounded by $ \epsilon^{\beta_0}$. In addition, for every fixed $2 \leq j \leq M$, since $ \{ 1,2, \ldots,K\} $ is accessible by $X_{j}$ by the definition of $K$, it is clear that $ \Pr{X_{j} \textrm{ reaches } [\epsilon n, \infty) \textrm { before }  [1,K]} \rightarrow 0$ as $n \rightarrow \infty$. The conclusion follows.
\end {proof}

\subsection {The absorbed case: proof of Theorem \ref {thm:main2}}
\label {sec:cvthm2}

Recall that $ {X}^{\dagger}_{n}$ denotes the Markov chain $X_{n}$ stopped when it hits $ \{ 1,2, \ldots,K\} $. As for the non-absorbed case, Theorem \ref {thm:main2} will follow if we manage to  establish that
 \begin{equation}
 \label{eq:case2}\left( \frac{ {X}^{\dagger} _{n}( \mathcal{N} _{n} (t) )}{n} ; t \geq 0 \right)   \quad\mathop{\longrightarrow}^{(d)}_{n \rightarrow \infty} \quad (Y_t; t\geq 0)
 \end{equation}
 in $\D(\R_{+},\R)$. 

We now need to introduce some additional notation. Fix $M \geq 1$ and set $a^{(M)}_{i}=a_{i}$ for $i>M$ and $a^{(M)}_{i}=0$ for $1 \leq i \leq M$.  Denote by $L_{n}^{(M)}$ the Markov chain with generator \eqref{eq:genLn} when the sequence $(a_{n})_{n \geq 1}$ is replaced with the sequence $(a^{(M)}_{n})_{n \geq 1}$. In other words, $L_{n}^{(M)}$ may be seen as $L_{n}$ absorbed at soon as it hits $ \{ \ln(1/n), \ln(2/n), \ldots, \ln(M/n) \}$.  Proposition \ref {prop:cvL} (applied with the sequence $(a^{(M)}_{n})$ instead of $(a_{n})$), shows that, under \textbf{(A1)} and \textbf{(A2)},  $L^{(M)}_{n}$, started from any $x \in \R$, converges in distribution in $ \D(\R_{+},\R)$ to $ \xi+x$.  In addition, 
if $L_{n}^{(M)}(0)=0$ and if $X_{n}^{(M)}$ denotes the process $X_{n}$ absorbed as soon as hits $ \{ 1,2, \ldots,M\}$,  Lemma \ref {lem:coupling}  (applied with $(a^{(M)}_{n})$ instead of $(a_{n})$) entails that 
\begin{equation}
\label{eq:egM}\left(  \frac{1}{n} X^{(M)}_{n}(  { \mathcal{N}_{n}(t) }) ; t \geq 0 \right)  \quad\mathop{=}^{(d)}  \quad \left( \exp \left( L^{(M)}_{n}(\tau^{(M)}_{n}(t)) \right) ; t \geq 0 \right),
\end{equation}
where
$$\tau^{(M)}_{n}(t)= \inf \left\{ u \geq  0 ; \int_{0}^{u} \frac{a^{(M)}_{n \exp(L_{n}(s))}}{a^{(M)}_{n}} \ ds >t \right\}, \qquad t \geq 0.
$$
In particular, 
\begin{equation}
\label{eq:eg1}\left(  \frac{1}{n} X^{\dagger}_{n}(  { \mathcal{N}_{n}(t) }) ; t \geq 0 \right)  \quad\mathop{=}^{(d)}  \quad \left( \exp(L^{(K)}_{n}(\tau^{(K)}_{n}(t))); t \geq 0 \right).
\end{equation}
Unless explicitly stated otherwise, we  always assume that $L_{n}^{(M)}(0)=0$.

In the sequel, we denote by $ \dsk$ the Skorokhod $J_{1}$ distance on $ \D(\R_{+}, \R)$.  In the proof of Theorem \ref {thm:main2}, we will use the following simple property of  $\dsk$:

\begin {lem}\label {lem:skorokhod}Fix $ \epsilon>0$ and $f \in \D( \R_{+},\R)$ that has limit $0$ at $ + \infty$. Let $ \sigma: \R_{+} \rightarrow \R_{+} \cup \{  + \infty\} $ be a right-continuous non-decreasing function. For $T \geq 0$, let $f^{[T]} \in \D( \R_{+},\R)$ be the function defined by $f^{[T]}(t)= f(\sigma(t) \wedge T)$ for $t \geq 0$. Finally, assume that there exists $T>0$ is such that  $|f(t)| <  \epsilon$ for every $t \geq T$. Then $ \dsk \left(f \circ \sigma,f^{[T]} \right)  \leq \epsilon$.
\end {lem}
This is a simple consequence of the definition of the Skorokhod distance.
We are now ready to complete the proof of Theorem \ref {thm:main2}.

\begin {proof}[Proof of Theorem \ref {thm:main2}] By \eqref{eq:eg1}, it suffices to check that 
\begin{equation}
\label{eq:mq2}\left( \exp\left( L^{(K)}_{n}(\tau^{(K)}_{n}(t)) \right) ; t \geq 0 \right)  \quad\mathop{\longrightarrow}^{(d)}_{n \rightarrow \infty} \quad (Y(t); t\geq 0)
\end{equation}
in $ \D( \R_{+},\R)$. To simplify notation, for $n \geq 1, t \geq 0$, set $Y^{\dagger}_{n}(t)= \exp\left( L^{(K)}_{n}( \tau^{(K)}_{n}(t)) \right) $, and, for every $t_{0}>0$, 
 $$Y^{t_{0}}_{n}(t)=\exp \left( L^{(K)}_{n}( \tau^{(K)}_{n}(t) \wedge t_{0}) \right) , \qquad Y^{t_{0}}(t)=\exp \left(  \xi( \tau(t) \wedge t_{0}) \right) $$
 and recall that $Y(t)=\exp \left(  \xi( \tau(t)) \right) $. 

First observe that for every fixed $t_{0}>0$,
\begin{equation}
\label{eq:cvtronquee}Y^{t_{0}}_{n} \quad\mathop{\longrightarrow}^{(d)}_{n \rightarrow \infty} \quad  Y^{t_{0}}
\end{equation}
in $\D(\R_{+},\R)$. 
 Indeed, since $L_{n}^{(K)} \rightarrow \xi$ in distribution in $ \D(\R_{+},\R)$, the same arguments as in Section \ref{sec:nonabsorbed} apply and give that $ \tau^{(K)}_{n}(\cdot)\wedge t_{0} \rightarrow \tau(\cdot) \wedge t_{0}$ in distribution in $ \mathcal{C}(\R_{+},\R)$. 
 
 We now claim that for every $ \eta \in (0,1)$, there exists $t_{0}>0$ such that for every $n$ sufficiently large,
\begin{equation}
\label{eq:borne} \Pr{\dsk \left( Y,  Y^{t_{0}} \right)> \eta} < 2\eta^{\beta_{0}}, \qquad  \Pr{\dsk  \left(   Y^{\dagger}_{n},  Y^{t_{0}}_{n} \right)> \eta} < 3\eta^{\beta_{0}}.
\end{equation}
Assume for the moment that \eqref{eq:borne}  holds and let us see how to finish the proof of \eqref{eq:mq2}.  Let $F: \D( \R_{+}, \R) \rightarrow \R_{+}$ be a bounded uniformly continuous function. By \cite[Theorem 2.1]{Bil99}, it is enough to check that $ \E {F (Y^{\dagger}_{n})} \rightarrow \Es {F(Y)}$ as $n \rightarrow \infty$. Fix $ \epsilon \in (0,1)$ and let $ \eta>0$ be such that $|F(f)-F(g)| \leq  \epsilon$ if $\dsk(f,g) \leq \eta$. We shall further impose that $\eta^{\beta_0}<\epsilon$. By \eqref{eq:borne}, we may choose $t_{0}>0$ such that 
the events
$$\Lambda= \left\{  \dsk \left( Y,  Y^{t_{0}} \right)< \eta \right\}, \qquad \Lambda_{n}=  \left\{ \dsk  \left(   Y^{\dagger}_{n},  Y^{t_{0}}_{n} \right)< \eta \right\} $$
 are both of probability at least $1-3\eta^{\beta_{0}}\geq 1-3\epsilon$ for every $n$ sufficiently large. Then write for $n$ sufficiently large
\begin{eqnarray*}
 \left| \Es {F(Y)}-  \Es{F (Y^{\dagger}_{n})} \right| & \leq & \left| \Es {F(Y) \mathbbm {1}_{  \Lambda }}-  \Es{F (Y^{\dagger}_{n}) \mathbbm {1}_{ \Lambda_{n} }} \right|+ 6 \epsilon \norme {F} \\
 & \leq & \left| \Es {F(Y^{t_{0}}) \mathbbm {1}_{  \Lambda }}-  \Es{F (Y^{t_{0}}_{n}) \mathbbm {1}_{ \Lambda_n }} \right|+2 \epsilon+ 6 \epsilon  \norme {F}\\
 & \leq & \left| \Es {F(Y^{t_{0}})}-  \Es{F (Y^{t_{0}}_{n})} \right|+2 \epsilon+ 12 \epsilon  \norme {F}.
\end{eqnarray*}
By \eqref{eq:cvtronquee},  $\left| \Es {F(Y^{t_{0}})}-  \Es{F (Y^{t_{0}}_{n})} \right|$ tends to $0$ as $n \rightarrow \infty$. As a consequence, $$ \left| \Es {F(Y)}-  \Es{F (Y^{\dagger}_{n})} \right| \leq 3 \epsilon+ 12 \epsilon  \norme {F}$$ for every $n$ sufficiently large. 

We finally need to establish  \eqref{eq:borne}. For the first inequality,    since $ \xi$ drifts to $ - \infty$, we may choose $t_{0}>0$ such that $ \Pr { \xi(t_{0})<2 \ln(\eta)} >1- \eta^{\beta_{0}}$. By Lemma \ref {lem:continuous} and the Markov property
$$ \P\left(  \sup_{s \geq t_0}  e^{\xi(s)-\xi(t_0)}> 1/\eta \right)  \leq  \eta ^{ \beta_{0}}.$$
The event $ \{ \sup_{s \geq t_{0}}  e^{\xi(s)} \leq  \eta\} $ thus has  probability at least $1-2\eta^{\beta_{0}}$, and on this event, we have $\dsk  (  Y,Y^{t_{0}} ) \leq  \eta$
by Lemma \ref {lem:skorokhod}. This establishes the first inequality of \eqref{eq:borne}.

For the second one, note that since $ L^{(K)}_{n}$ converges in distribution to $ \xi$, there exists $t_{0}>0$ such that $ \Pr {\exp(L^{(K)}_{n}(t_{0}))> \eta^{2}}< \eta$ for every $n$ sufficiently large.  But on the event
$$ \left\{ \exp(L^{(K)}_{n}(t_{0}))< \eta^{2} \right\} \cap \left\{  \textrm {after time } t_{0}, n\exp(L_{n}) \textrm{ reaches } [1,K] \textrm { before }  [\eta n, \infty) \right\},$$ 
which has probability at least $1-3 \eta^{\beta_{0}}$ by Lemma \ref {lem:discrete} (recall also the identity \eqref{eq:eg1}), we have  the inequality
$\dsk  (   Y^{\dagger}_{n},  Y^{t_{0}}_{n} ) \leq  \eta$ by Lemma \ref{lem:skorokhod}.
This establishes \eqref{eq:borne} and completes the proof of Theorem \ref {thm:main2}.
\end {proof}

\subsection {Convergence of the absorption time}
\label {sec:abs}

We start with several preliminary remarks in view of proving Theorem \ref {thm:absorption}.
First, we point out that our statements in Section \ref {sec:description} are unchanged if we replace the sequence $(a_n)$ by another sequence, say $(a'_n)$, such that $a_n/ a'_n \rightarrow1$ as $n \rightarrow \infty$.
Thanks to Theorem 1.3.3 and Theorem 1.9.5 (ii) in \cite{BGT87}, we may therefore assume that there exists an infinitely differentiable function $h: \R_{+} \rightarrow \R$  such that
\begin{equation}\label{eq:diffa}
(i) \textrm { for every } n \geq 1, \quad a_{n}=n^{\gamma} \cdot  e^{h(\ln(n))}, \qquad (ii) \textrm { for every } k \geq 1, \quad h^{(k)}(x)  \quad\mathop{\longrightarrow}_{x \rightarrow \infty} \quad 0,
\end{equation}
where $h^{(k)}$ denotes the $k$-th derivative of $h$. This will be used in the proof of Lemma \ref {lem:cvdom} below.

Assume that \textbf{(A1)}, \textbf{(A2)}, \textbf{(A3)} hold and that $ \xi$ drifts to $- \infty$.  For every integer $M \geq 1$, recall from Section \ref {sec:cvthm2} the notation $X^{(M)}_{n}$, $(a^{(M)}_{n})$ and $L^{(M)}_{n}$, and the initial condition $L^{(M)}_{n}(0)=0$. To simplify the notation, we set $ \widetilde{a}_{n}=a^{(M)}_{n}$ for $n \geq 1$ and $ \widetilde{L}_{n}(s)=L^{(M)}_{n}(s)$. By \eqref{eq:egM}, we may and will assume that the identity
$$\frac{1}{n} X^{(M)}_{n}(  { \mathcal{N}_{n}(t) }) \ = \ \exp \left( L^{(M)}_{n}(\tau^{(M)}_{n}(t)) \right) $$
holds for all $t\geq 0$, where $ \mathcal{N}_{n}$ is a Poisson process with intensity $a_{n}$ independent of $X_{n}$ and the time change $\tau^{(M)}_{n}$ is defined by \eqref{eq:timechange} with $ \widetilde{a}_{n}=a^{(M)}_{n}$ replacing $a_n$.

For $n>M$, let $A^{(M)}_{n}= \inf \{ i \geq 1; X_{n}(i) \leq M \} $ be the absorption time of $X^{(M)}_{n}$  and ${\alpha}^{(M)}_{n}= \inf \{t \geq 0; {X} _{n}( \mathcal{N} _{n}( t)) \leq M \}$
that of $X^{(M)}_{n}(  { \mathcal{N}_{n}(\cdot) })$,  so that there are the identities
  \begin{equation}
  \label{eq:edist} {{\alpha}^{(M)}_{n}}  \ = \  \int_{0}^{\infty} \frac{\widetilde {a}_{n \exp( \widetilde{L}_{n}(s)) } 
}{\widetilde{a}_{n}}ds = \int_{0}^{{\alpha}^{(M)}_{n}} \frac{ {a}_{n \exp( {L}_{n}(s)) } 
}{{a}_{n}}ds\quad \hbox{and} \quad \mathcal{N} _{n} \left(  {\alpha}^{(M)}_{n} \right)  = A^{(M)}_{n}
  \end{equation}
  for every $n>M$. We shall first establish a weaker version of Theorem \ref {thm:absorption} (i) in which $K$ has been replaced by $M$:

\begin{lem}\label{lem:add} Assume that \textbf{(A1)}, \textbf{(A2)}, \textbf{(A3)} hold and that $ \xi$ drifts to $- \infty$. The following weak convergences hold jointly in $\D(\R_{+},\R) \otimes \R$:
 $$ \widetilde{L}_{n} \quad\mathop{\longrightarrow}^{(d)}_{n \rightarrow \infty} \quad  \xi \qquad \hbox{and} \qquad
{{\alpha}^{(M)}_{n}}   \quad\mathop{\longrightarrow}^{(d)}_{n \rightarrow \infty} \quad  \int_{0}^{\infty} e^{\gamma \xi(s)}ds.$$
\end{lem}

In turn, in order to establish Lemma \ref{lem:add}, we shall need the following technical result:

\begin {lem}\label {lem:cvdom} Assume that \textbf{(A1)}, \textbf{(A2)}, \textbf{(A3)} and \eqref{eq:diffa} hold, and that $ \xi$ drifts to $- \infty$. There exist $ \beta>0$, $M>0$ and $C>0$ such that for every $n \geq M$,
\begin{equation}
\label{eq:a}\sum_{k \geq 1}(a_{k}^{\beta}-a_{n}^{\beta}) \cdot p_{n,k} \leq -C \cdot a_{n}^{\beta-1}.
\end{equation}
\end {lem}
If $ a_{n}= c \cdot n^{\gamma} $ for every $n$ sufficiently large for a certain $c>0$, observe that this is a simple consequence of \eqref{eq:claim} applied with $ \beta_{0}= \beta \gamma$. Note also that 
\eqref{eq:a} then clearly holds when $(a_{n})$ is replaced with $(\widetilde {a}_{n})$. 
We postpone its proof in the general case to the end of this section.

\begin{proof} [Proof of Lemma \ref{lem:add}] The first convergence has been established in the proof of Proposition \ref{prop:cvL}. Using Skorokhod's representation Theorem, we may assume that it holds in fact almost surely on $\D(\R_+,\R)$, and we shall now check that this entails the second. 
 To this end, note first that  for every $R \geq 0$, 
$$  \int_{0}^{R} \frac{\widetilde {a}_{n \exp( \widetilde{L}_{n}(s)) }
}{\widetilde{a}_{n}}ds  \quad\mathop{\longrightarrow}^{\rm a.s.}_{n \rightarrow \infty} \quad  \int_{0}^{R} e^{\gamma \xi(s)} ds,$$
since the sequence $(a_n)$ varies regularly with index $\gamma$. 
It is therefore enough to check that for every $ \epsilon>0$ and $ t>0$, we may find $R$ sufficiently large so that
\begin{equation}
\label{eq:check1} \limsup_{n \rightarrow \infty} \Pr { \int_{R}^{\infty} \frac{\widetilde {a}_{n \exp( \widetilde{L}_{n}(s)) }
}{\widetilde{a}_{n}}ds > t} \leq  \epsilon \qquad \textrm {and} \qquad\Pr {\int_{R}^{\infty} e^{\gamma \xi(s)} ds>t} \leq  \epsilon.
\end{equation}
The second inequality is obvious since $\int_{0}^{\infty} e^{\gamma \xi(s)} ds$ is almost surely finite. 

To establish the first inequality in \eqref{eq:check1}, we start with some preliminary observations. By the Potter bounds (see \cite[Theorem 1.5.6]{BGT87}), there exists a constant $C_{1}>0$ such that $ \widetilde{a}_{i}/\widetilde{a}_{n} \leq C_{1} (i/n)^{\gamma+1}$ for every $1 \leq i \leq n$. Fix $ \eta>0$ such that $ 2^{\beta+1}C_{2} C_{1}^{\beta}  \eta^{\beta( \gamma+1)}/t^{\beta}< \epsilon$, where $C_{2}$ is a positive constant (independent of $ \eta$ and $ \epsilon$) which will be chosen later on.  Then pick  $R$ sufficiently large so that   $ \Pr {\exp(\widetilde{L}_{n}(R))> \eta  }< \epsilon/2$ for every $n$ sufficiently large (this is possible since $ \widetilde{L}_{n}$ converges  to $ \xi$ and the latter drifts to $-\infty$). By the Markov property and \eqref{eq:edist}, for every $i \geq 1$, the conditional law of 
$$ \int_{R}^{\infty} \frac{\widetilde {a}_{n \exp( \widetilde{L}_{n}(s)) }
}{\widetilde{a}_{n}}ds $$
given $n\exp(\widetilde{L}_{n}(R))=i$, is that of $\alpha^{(M)}_{i}$.
 It follows from \eqref{eq:edist} and elementary estimates for Poisson processes that is suffices to check 
\begin{equation}
\label{eq:ch}  \limsup_{n \rightarrow \infty}\max_{M+1 \leq i \leq  \eta  n} \Pr { A_{i}^{(M)}>t \widetilde{a}_{n}/2} \leq  \epsilon/2.
\end{equation}
To this end, for every  $i \geq M+1$ and $n \geq 1$, we use Markov's inequality and get
\begin{equation}
\label{eq:born} \Pr { A_{i}^{(M)}>t \widetilde{a}_{n}/2} \leq  \frac{2^{\beta}}{ t^{\beta} \widetilde{a}_{n}^{\beta}}   \Es {(A_{i}^{(M)})^{\beta}}.
\end{equation} 
We then apply Theorem 2' in \cite{AI98} with
$$f(x)=x^{\beta}, \quad h(x)=\widetilde{a}_{x}^{\beta},  \quad g(x)=\widetilde{a}_{x}^{\beta-1},$$
which tells us that there exists a constant $C_{2}>0$ such that $ \Es {f(A_{i}^{(M)})} \leq C_{2} \cdot  h(i)$ for every $i \geq M+1$, provided that we check the existence of a constant $C>0$ such that the two conditions
$$ \Es {h(X_{n}(1))-h(n)} \leq - C \cdot g(n) \quad \textrm {for every } n \geq M, \qquad \textrm{and} \qquad \liminf_{n \rightarrow \infty} \frac{g(n)}{f' \circ f^{-1} \circ h(n)}>0$$
hold.
This first condition follows from \eqref{eq:a}, and for the second, simply note that  we have 
$$ {g(n)}/({f' \circ f^{-1} \circ h(n)})=1/\beta.$$ By \eqref{eq:born}, we therefore get that for every $i \geq M+1$ and $n \geq 1$,
 $$ \Pr { A_{i}^{(M)}>t \widetilde{a}_{n}/2} \leq  C_{2} \frac{2^{\beta}}{t^{\beta}} \, \left( \frac{\widetilde{a}_{i}}{  \widetilde{a}_{n}} \right) ^{\beta} .$$
 As a consequence of the aforementioned Potter bounds, for every $M+1 \leq i \leq  \eta n$,
 $$ \Pr { A_{i}^{(M)}>t \widetilde{a}_{n}/2} \leq \frac{2^{\beta} C_{2} C_{1}^{\beta}}{t^{\beta}} \cdot \eta^{\beta ( \gamma+ 1)}< \epsilon/2.$$
 This entails \eqref{eq:ch}, and completes the proof. 
\end{proof}

We are now ready to start the proof of Theorem \ref {thm:absorption}.

\begin {proof}[Proof of Theorem \ref{thm:absorption}]  (i) Assume that $M \geq K$, $ \beta> 0$, $c_{0}>0$ are such that Lemma \ref {lem:cvdom} and Lemma \ref {lem:utile} (iii) hold (with $ \beta$ instead of $ \beta_{0}$).  

 It suffices to check that \eqref {eq:cvtemps} holds with $A_{n}^{(K)}$ replaced by ${A}^{(M)}_{n}$. Indeed, since $ {X}^{(M)}_{n}$ and $ {X}^{\dagger}_{n}$ may be coupled in such a way that they coincide until the first time $X_{n}$ hits $ \{ 1,2, \ldots,M\}$,  for every $a>0$ we have
 $$  \Pr { \left |A^{(K)}_{n}-  {A}^{(M)}_{n} \right |>a} \leq  \max_{K+1 \leq i \leq M} \Pr {A^{(K)}_{i}>a} $$ which tends to $0$ as $a \rightarrow \infty$ by Lemma \ref {lem:utile} (iv). In turn, as before, since $(\mathcal{N} _{n}(t)/ a_{n}; t \geq 0)$ converges in probability to the identity uniformly on compact sets  as $n \rightarrow \infty$, it is enough to check that the convergence
$$ {\widetilde{A}^{(M)}_{n}} \quad\mathop{\longrightarrow}^{(d)}_{n \rightarrow \infty} \quad \int_{0}^{\infty} e^{\gamma \xi(s)}ds$$ holds jointly with \eqref{eq:case2}. 
By the preceding discussion and \eqref{eq:edist}, we can complete the proof with an appeal to Lemma \ref{lem:add}.

(ii)  Again, it suffices to check that \eqref {eq:CVL1} holds with $A^{(K)}_{n}$ replaced by ${A}^{(M)}_{n}$. Indeed, we see from Markov property that 
$$ \Es{ \left| A_{n}^{(K)}-  {A}^{(M)}_{n} \right| }\leq  \max_{K+1 \leq i \leq M} \Es{A^{(K)}_{i}},$$ and the right-hand side is finite by Lemma \ref {lem:utile2} (ii). 

Recall that ${\mathcal N}_n$ is a Poisson process with intensity $a_n$, 
so by \eqref{eq:edist}, we have for $n>M$
$$\frac{1}{a_n} \Es{A_n^{(M)}} = \Es{\alpha^{(M)}_n} =   \Es{ \int_{0}^{{\alpha}^{(M)}_{n}} \frac{ {a}_{n \exp( {L}_{n}(s)) } 
}{{a}_{n}}ds} 
$$
and we thus have to check that 
\begin{equation}
\label{eq:esp} \int_{0}^{\infty}  \Es{\frac{ {a}_{n \exp( {L}_{n}(s)) } 
}{{a}_{n}} \mathbbm{1}_{\{s< {\alpha}^{(M)}_{n}\}}} ds  \quad\mathop{\longrightarrow}_{n \rightarrow \infty} \quad  \int_{0}^{\infty}  \Es{e^{\gamma \xi(s)}} ds = \frac{1}{|\Psi(\gamma)|}.
\end{equation}

In this direction, take any  $\beta\in(\gamma, \beta_0)$, and recall from 
Potter bounds \cite[Theorem 1.5.6]{BGT87} that there is some constant $C>0$ such that 
$ {{a}_{n}}^{-1}{{a}_{nx}} \leq   C \cdot x^{ \beta}$ for every $n\in\N$ and $x \geq 0$ with $nx \in \Z_{+}$.
We deduce that
$$
\Es{ \left( \frac{{a}_{n \exp( {L}_{n}(s)) }}{a_{n}} \right) ^{\beta_0/\beta} \mathbbm{1}_{ \left\{s< {\alpha}^{(M)}_{n} \right\}}} \leq  C^{\beta_0/\beta}\cdot \Es{\exp(\beta_0  L_n(s)) \mathbbm{1}_{ \left\{s< {\alpha}^{(M)}_{n} \right\}}} 
\leq  C'  \cdot e^{-cs},
$$
where $c, C'$ are positive finite constants, and the last inequality stems from Corollary \ref{cor:super}. 
Then recall that ${{a}_{n}}^{-1} 
{ {a}_{n \exp( {L}_{n}(s)) }}  \mathbbm{1}_{\{s< {\alpha}^{(M)}_{n}\}}$ converges in distribution to $\exp(\gamma \xi(s))$ for every $s\geq 0$. An argument of uniform integrability now shows that \eqref{eq:esp} holds, and this completes the proof. 
\end {proof}

\begin {rem} \label {rem:moments}The argument of the proof above shows that more precisely,  for every $1 \leq p < \beta_{0}/\gamma$, we have
$$ \Es{ \left( \frac{A^{(M)}_{n}}{a_n} \right) ^{p}} \quad\mathop{\longrightarrow}_{n \rightarrow \infty} \quad  \Es{ \left( \int_{0}^{\infty} e^{\gamma \xi(s)}ds \right) ^{p}}.$$
\end {rem}

\begin {rem}\label {rem:inf} Assume that \textbf{(A1)}, \textbf{(A2)} and \textbf{(A4)} hold. Let $ 1 \leq  m \leq K$ be an integer. Then $$   \Es {A_{m}^{(K)}}= \infty \quad \Longleftrightarrow  \quad \sum_{k \geq 1} a_{k} \cdot p_{m,k} = \infty.$$
 Indeed, by the Markov property applied at time $1$, write
$$ \Es {A_{m}^{(K)}}=1+ \sum_{k \geq K+1} \Es {A_{k}^{(K)}} p_{m,k}.$$
By Lemma \ref {lem:utile2} (ii), $\E {A_{k}^{(K)}} < \infty$ for every $k \geq K+1$, and by Theorem \ref {thm:absorption} (ii),  $ \E {A_{k}^{(K)}}/ a_{k}$ converges to a positive real number as $k \rightarrow \infty$. Therefore, there exists a constant $C>0$ such that $ a_{k}/C \leq  \E {A_{k}^{(K)}} \leq  C \cdot  a_{k}$ for every $k \geq K+1$. As a consequence,
$$  \frac{1}{C}(\E {A_{m}^{(K)}}-1) =  \frac{1}{C}  \sum_{k \geq K+1} \E {A_{k}^{(K)}} \cdot p_{m,k} \leq   \sum_{k \geq K+1} a_{k} \cdot p_{m,k} \leq C  \sum_{k \geq K+1} \E {A_{k}^{(K)}} p_{m,k} =  C(\E {A_{m}^{(K)}}-1).$$
The conclusion follows.
\end {rem}

We conclude this section with the proof of Lemma \ref {lem:cvdom}.

\begin {proof}[Proof of Lemma \ref {lem:cvdom}]

 By Lemma \ref {lem:utile}, there exists $ \beta_{0}>0$ such that $ \Psi(\beta_{0})<0$. Fix $ \beta<  \beta_{0} \wedge (\beta_{0}/ \gamma)$ and note that $ \Psi(\beta\gamma)<0$ by convexity of $\Psi$. We shall show that 
\begin{equation}
\label{eq:C}  a_{n}^{1-\beta} \sum_{k \geq 1} (a_{k}^{\beta}-a_{n}^{\beta}) \cdot  p_{n,k}  \quad\mathop{\longrightarrow}_{n \rightarrow \infty} \quad\Psi(\beta \gamma).
\end{equation}
To this end,  write 
$$a_{n}^{1-\beta} \sum_{k \geq 1} (a_{k}^{\beta}-a_{n}^{\beta})= a_{n} \int_{-\infty}^{\infty}  \left(  e^{\beta\gamma x}-1 \right)  \Pi_{n}^{\ast}(dx)+ a_{n} \int_{-\infty}^{\infty}  \left(   \left(  \frac{a_{ne^{x}}}{a_{n}} \right) ^{\beta}- e^{\beta \gamma x} \right)  \Pi_{n}^{\ast}(dx).$$
By Lemma \ref {lem:utile} (ii), the result will follow if we prove that
$$a_{n} \int_{-\infty}^{\infty}  \left(   \left(  \frac{a_{ne^{x}}}{a_{n}} \right) ^{\beta}- e^{\beta \gamma x} \right)  \Pi_{n}^{\ast}(dx)  \quad\mathop{\longrightarrow}_{n \rightarrow \infty} \quad 0.$$
In this direction, we first check that
\begin{equation}
\label{eq:interm0}a_{n} \int_{x \geq 1}  \left(   \left(  \frac{a_{ne^{x}}}{a_{n}} \right) ^{\beta}- e^{\beta \gamma x} \right)  \Pi_{n}^{\ast}(dx)  \quad\mathop{\longrightarrow}_{n \rightarrow \infty} \quad 0, \qquad a_{n} \int_{x \leq - 1}  \left(   \left(  \frac{a_{ne^{x}}}{a_{n}} \right) ^{\beta}- e^{\beta \gamma x} \right)  \Pi_{n}^{\ast}(dx)  \quad\mathop{\longrightarrow}_{n \rightarrow \infty} \quad 0.
\end{equation}
 By standard properties of regularly varying functions (see e.g.~\cite[Theorem 1.5.2]{BGT87}), $ ({a_{ne^{x}}}/{a_{n}}) ^{\beta}$ converges to $e^{\beta \gamma x}$ as $n \rightarrow \infty$, uniformly in $x \leq -1$. By \textbf{(A1)} and \eqref{eq:condPi}, this readily implies the second convergence of \eqref{eq:interm0}. For the first one, a similar argument shows that the convergence of $({a_{ne^{x}}}/{a_{n}}) ^{\beta}$  to $e^{\beta \gamma x}$ as $n \rightarrow \infty$ holds uniformly in $x \in [1,A]$, for every fixed $A>1$. Therefore, if $ \eta>0$ is fixed, it is enough to establish the existence of $A>1$ such that
 \begin{equation}
 \label{eq:interm2}\limsup_{n \rightarrow \infty}a_{n} \int_{A}^{\infty}  \left|   \left(  \frac{a_{ne^{x}}}{a_{n}} \right) ^{\beta}- e^{\beta \gamma x} \right|  \Pi_{n}^{\ast}(dx)  \leq  \eta.
 \end{equation}
To this end, fix $ \epsilon>0$ such that $ \beta(\gamma+ \epsilon)< \beta_{0}$. By the Potter bounds, there exists a constant $C>0$ such that for every $x \geq 1$ and $n \geq 1$ we have
$$\left|   \left(  \frac{a_{ne^{x}}}{a_{n}} \right) ^{\beta}- e^{\beta \gamma x} \right| \leq C e^{ \beta(\gamma+\epsilon)x}+ e^{\beta \gamma x}.$$
Since  $ \int_{1}^{\infty}e^{\beta(\gamma+\epsilon) x}\ \Pi(dx)< \infty$ and $ \int_{1}^{\infty}e^{\beta \gamma x}\ \Pi(dx)< \infty$ by our choice of $ \beta$ and $ \epsilon$, we may choose $A>0$ such that 
$$C \int_{A}^{\infty} e^{\beta(\gamma+\epsilon) x}\ \Pi(dx)+\int_{A}^{\infty} e^{\beta\gamma x}\ \Pi(dx)< \eta.$$
Hence
$$ \limsup_{n \rightarrow \infty}a_{n} \int_{A}^{\infty}  \left|   \left(  \frac{a_{ne^{x}}}{a_{n}} \right) ^{\beta}- e^{\beta \gamma x} \right|  \Pi_{n}^{\ast}(dx)  \leq C \int_{A}^{\infty} e^{\beta(\gamma+\epsilon) x} \Pi(dx)+\int_{A}^{\infty} e^{\beta \gamma x} \Pi(dx)< \eta.$$
This establishes \eqref{eq:interm2} and completes the proof of \eqref{eq:interm}.

We now show that
\begin{equation}
\label{eq:interm}a_{n} \int_{-1}^{1}  \left(   \left(  \frac{a_{ne^{x}}}{a_{n}} \right) ^{\beta}- e^{\beta \gamma x} \right)  \Pi_{n}^{\ast}(dx)  \quad\mathop{\longrightarrow}_{n \rightarrow \infty} \quad 0.
\end{equation}
By (i) in \eqref{eq:diffa}, we have 
$$\left(  \frac{a_{ne^{x}}}{a_{n}} \right) ^{\beta}- e^{\beta \gamma x}= e^{\beta \gamma x} \left( e^{ \beta h(\ln(n)+x)-\beta h(\ln(n))} -1\right).$$ 
For every $n \geq 1$ and $ x \in (-1,1)$, an application of Taylor-Lagrange's formula yields the existence of a real number $u_{n}(x) \in ( \ln(n)-1, \ln(n)+1)$ such that
$$h(\ln(n)+x)= h( \ln(n))+ x h^{(1)}( \ln(n))+ x^{2} h^{(2)}(u_{n}(x))/2,$$
where we recall that $h^{(k)}$ denotes the $k$-th derivative of $h$. Recalling (ii) in \eqref{eq:diffa}, we can write
$$e^{\beta \gamma x} \left( e^{ \beta h(\ln(n)+x)-\beta h(\ln(n))} -1\right)=  \beta x h^{(1)}( \ln(n))+ x^{2} g_{n}(x),$$
where $g_{n}(x) \rightarrow 0$ as $n \rightarrow \infty$, uniformly in $x \in (-1,1)$. Also note that  $ h^{(1)}(\ln(n))  \rightarrow 0$ as $n \rightarrow \infty$. Now,
\begin{equation}
\label{eq:sum}a_{n} \int_{-1}^{1}  \left(   \left(  \frac{a_{ne^{x}}}{a_{n}} \right) ^{\beta}- e^{\beta \gamma x} \right)  \Pi_{n}^{\ast}(dx)= \beta h^{(1)}(\ln(n)) \cdot a_{n} \int_{-1}^{1} x \ \Pi_{n}^{\ast}(dx)+ a_{n}\int_{-1}^{1} x^{2} g_{n}(x) \ \Pi_{n}^{\ast}(dx).
\end{equation}
By \textbf{(A2)} and the preceding observations, the sum appearing in \eqref{eq:sum} tends to $0$ as $n \rightarrow \infty$. This completes the proof.
\end {proof}

\subsection {Scaling limits for the non-stopped process}
\label {sec:nonstopped}
Here we establish Theorem \ref {thm:main3}.

\begin {proof}[Proof of Theorem \ref {thm:main3}]
By Theorem \ref {thm:main2}, Lemma \ref{lem:skorokhod} and the strong Markov property, it is enough to show that for every fixed $t_{0} >0$, $ \epsilon>0$ and $1 \leq i \leq K$, we have
\begin{equation}
\label{eq:mqthm3} \Pr { \sup_{0 \leq t \leq t_{0}} {X} _{i}( \lfloor a_{n} t \rfloor) \geq  \epsilon n }
=\Pr { \sup_{0 \leq k \leq \lfloor a_n t_{0}\rfloor } {X} _{i}( k) \geq  \epsilon n }  \quad\mathop{\longrightarrow}_{n \rightarrow \infty} \quad 0,
\end{equation}
To this end, fix $1 \leq i \leq K$, and introduce the successive return times to $ \{ 1,2, \ldots,K\} $ by $X_{i}$: 
$$T^{(1)}=A_i^{(K)} =\inf \{ j > 0; X_{i}(j)  \leq K\},$$
 and recursively, for $k \geq 2$,  
 $$T^{(k)}= \inf \{ j  > T^{(k-1)}; X_{i}(j) \leq K\}.$$ 
 Plainly, $T^{(k)}\geq k$ and we see from the strong Markov property that \eqref{eq:mqthm3} will follow if we manage to check that, for every $1 \leq i \leq K$,
\begin{equation}
\label{eq:mqthm3bis} a_{n} \cdot \Pr { \sup_{0 \leq j \leq T^{(1)}} {X} _{i}(j) \geq  \epsilon n }  \quad\mathop{\longrightarrow}_{n \rightarrow \infty} \quad 0.
\end{equation}
 
 To this end, introduce $ \tau_{n}= \inf \{ j \geq 1; X_{i}(j)> \epsilon n\}  \wedge T^{(1)}$ and note that $ \Es {\tau_{n}} \rightarrow \E {T^{(1)}}$ as $n \rightarrow \infty$ by monotone convergence since $ \{ 1,2, \ldots,K\}$ is accessible by $X_{n}$ for every $n \geq 1$. In addition,
$$ \Es {T^{(1)}- \tau_{n}}= \sum_{j \geq  \epsilon n}  \Pr {X_{1}(\tau_{n})=j} \Es {A^{(K)}_{j}},$$
But the last part of Theorem \ref {thm:absorption} shows that $\E {A^{(K)}_{j}}/a_{j}$ converges to some positive real number as $j \rightarrow \infty$ and thus  $\E{A^{(K)}_{j}} \geq C a_{j}$ for every $j \geq 1$ and some constant $C>0$.  Since $ \Es {\tau_{n}} \rightarrow \E {T^{(1)}}$ as $n \rightarrow \infty$, this implies that
$$ \sum_{j \geq  \epsilon n}  \Pr {X_{1}(\tau_{n})=j} a_{j}  \quad\mathop{\longrightarrow}_{n \rightarrow \infty} \quad 0.$$
In addition, by the Potter bounds, for $\eta>0$ arbitrary small, there exists a constant $C'>0$ such that $a_{j}/a_{n} \geq C' (j/n)^{\gamma-\eta} \geq  C' \epsilon^{\gamma- \eta}$ for every $n \geq 1$ and $j \geq  \epsilon n$. Therefore
$$a_{n} \cdot \sum_{j \geq  \epsilon n}  \Pr {X_{i}(\tau_{n})=j}   \quad\mathop{\longrightarrow}_{n \rightarrow \infty} \quad 0,$$
which is exactly \eqref{eq:mqthm3bis}. This completes the proof. 
\end {proof}

\section {Applications}
\label {sec:applications}

We shall now illustrate our general results stated in Section \ref {sec:description} by discussing some special cases which may be of independent interest. Specifically, we shall first show how one can  recover the results of Haas \& Miermont \cite{HM11} about the scaling limits of decreasing Markov chains, then we shall discuss limit theorems for Markov chains with asymptotically zero drift. Finally, we shall apply our results to the study of the number of blocks in some exchangeable fragmentation-coagulation processes (see \cite{Ber04}). 

\subsection {Recovering previously known results}
\label {sec:HM}

Let us first explain how to recover the result of Haas \& Miermont. For $n \geq 1$, denote by  $p_{n}^{\ast}$  the probability measure on $ \R_{+}$ defined  by
$$p_{n}^{\ast}(dx)= \sum_{k \geq 1}  p_{n,k} \cdot \delta_{ \frac{k}{n}}(dx),$$ 
which is the law of $\frac{1}{n} X_{n}(1)$. In \cite {HM11}, Haas \& Miermont establish the convergence \eqref{thm:main2} under the assumption of the existence of a non-zero, finite, non-negative measure $ \mu$ on $[0,1]$ such that the convergence
\begin{equation}
\label{eq:condHM} a_{n}(1-x) \cdot p_{n}^{\ast}(dx)  \quad\mathop{\longrightarrow}^{(w)}_{n \rightarrow \infty} \quad  \mu(dx)
\end{equation}
holds for the weak convergence of measures on $[0,1]$. Our framework covers this case, where the limiting process $Y$ is decreasing. Indeed,  assuming \eqref{eq:condHM} and $ \mu (\{0\})=0$ (i.e.~there is no killing), let $\widetilde{\mu}$ be the image of $ \mu$ by the mapping $x \mapsto \ln(x)$, and let $ \Pi(dx)$ be the measure $\widetilde{\mu}(dx)/(1-e^{x})$, which is supported on $(-\infty,0)$ (the image of $ \Pi(dx)$ by $x \mapsto -x$ is exactly the measure $ \omega(dx)$ defined in \cite[p.~1219]{HM11}). Then:

\begin {prop} Assume \eqref{eq:condHM} with  $ \mu (\{0\})=0$. We then have $\int_{- \infty}^{\infty} (1 \wedge |x|)  \ \Pi(dx) < \infty$ and  \textbf{(A1)}, \textbf{(A2)} hold with
$$b= \int_{-1}^{0} x \ \Pi(dx)+ \mu( \{ 1\}) = \int_{1/e}^{1}  \frac{ \ln(x)}{1-x} \ \mu(dx) + \mu( \{ 1\}), \qquad \sigma^{2}=0.$$ 
In addition, \textbf{(A3)}, \textbf{(A4)} and \textbf{(A5)} hold for every $ \beta>0$.\end {prop}
 
\begin{proof}
 This simply follows from the facts that for every continuous bounded function $f: \R \rightarrow \R_{+}$,
 $$ \int_{-\infty}^{\infty} f(x) \ \Pi(dx)= \int_{0}^{1} \frac{f( \ln(x))}{1-x} \ \mu(dx), \qquad  a_{n} \int_{-\infty}^{\infty} f(x) \ \Pi^{\ast}_{n}(dx)= \int_{0}^{1} \frac{f(\ln(x))}{1-x} \cdot a_{n}(1-x) \ p_{n}^{*}(dx)$$
 and that, as noted in \cite[p.~1219]{HM11}),
 $$ \Psi(\lambda)=- \mu( \{ 1\}) \cdot \lambda + \int_{-\infty}^{0}(e^{\lambda x}-1) \Pi(dx),$$ 
which is negative for every $ \lambda>0$.
\end{proof} 

Then Theorem \ref{thm:main3} enables us to recover Theorem 1 in \cite{HM11}, whereas Theorem \ref{thm:absorption} yields the essence of Theorem 2 of \cite{HM11}.

We also mention that our results can be used to (partially) recover the invariance principles for random walks conditioned to stay positive due to Caravenna \& Chaumont \cite{CC08}, but we do not enter into details for the sake of the length of this article.
The interested reader is referred 
to the first version of this paper available on ArXiV for a full argument.

\subsection{Markov chains with asymptotically zero drift}
\label {sec:zd}

For every $n \geq 1$, let $ \Delta_{n}=X_{n}(1)-n$ be the first jump of the Markov chain $X_{n}$.  We say that this Markov chain has asymptotically zero drift if $ \Es {\Delta_{n}} \rightarrow 0$ as $n \rightarrow \infty$. The study of processes with asymptotically zero drift was initiated by Lamperti in  \cite {Lam60,Lam62b,Lam63}, and was continued by many authors; see \cite {Ale11} for a thorough bibliographical description.

A particular instance of such Markov chains are the so-called Bessel-type random walks, which are random walks on $ \N$, reflected at $1$, with steps $ \pm 1$ and transition probabilities
\begin{equation}
\label{eq:besseltype}p_{n,n+1}=p_{n}= \frac{1}{2} \left( 1- \frac{d}{2n}+ o \left(  \frac{1}{n} \right)  \right) \quad \textrm { as } {n \rightarrow \infty}, \qquad p_{n,n-1}=q_{n}=1-p_{n},
\end{equation}
where $d \in \R$. The study of Bessel-type random walks has attracted a lot of attention starting from the 1950s in connection with birth-and-death processes, in particular concerning the finiteness and local estimates of first return times  \cite {Har52,HR53,Lam62b,Lam63}; see also the Introduction of \cite {Ale11}, which contains a concise and precise bibliographical account. Also, the interest to Bessel-type random walks has been recently renewed due to their connection to
statistical physics models such as random polymers \cite {DDH09,Ale11} (see again the Introduction of \cite {Ale11} for details) and a non-mean field model of coagulation--fragmentation \cite {BT12}. Non-neighbor Markov chains with asymptotically zero drifts have also appeared in \cite {MVW08} in connection with random billiards.

\begin{framed}
Assume that there exist $p>2$, $ \delta>0$, $C>0$ such that for every $n \geq 1$
\begin{equation}
\label{eq:zd1} \Es {|\Delta_{n}|^{p}} \leq C \cdot n^{p-2- \delta}.
\end{equation}
 Also assume that as $n \rightarrow \infty$,
\begin{equation}
\label{eq:zd2} \Es {\Delta_{n}}= \frac{c}{n}+ o \left(  \frac{1}{n} \right) , \qquad  \Es {\Delta_{n}^{2}}=s^{2}+o \left(1\right)
\end{equation}
for some $c \in \R$ and $s^{2} \in (0, \infty)$.
\end {framed}

 Finally, set 
$$r \quad = \quad- \frac{2c}{s^{2}}, \qquad \qquad \nu \quad =  \quad - \frac{1+r}{2}, \qquad \qquad  \delta \quad = \quad 1-r.$$
 Note that we do not require the Markov chain to be irreducible.

This model has been introduced and studied in detail in \cite {HMW13} (note however that in the latter reference, the authors impose the stronger conditions $\Es {\Delta_{n}}= \frac{c}{n}+ o (  (n \log(n))^{-1} ) $ and  $\Es {\Delta_{n}^{2}}=s^{2}+o (  {\log(n)}^{-1} )$ and also that the Markov chain is irreducible, but do not restrict themselves to the Markovian case).    

Note that Bessel-type random walks satisfying \eqref{eq:besseltype}  verify \eqref{eq:zd1} \& \eqref{eq:zd2} with $c=-d/2$ and $s=1$, so that $r=d$. 

In the seminal work \cite {Lam62b},  when $r<1$, under the additional assumptions that $\sup_{n \geq 1} \Es {|\Delta_{n}|^{4}}< \infty$ and that the Markov chain is uniformly null (see \cite {Lam62b} for a definition), Lamperti showed that $ \frac{1}{n}X_{n}$, appropriately scaled in time, converges in $ \D(\R_{+},\R)$ to a Bessel process. However, the majority of the subsequent work concerning Markov chains with asymptotically zero drifts and Bessel-type random walks was devoted to the study of the asymptotic behavior of return times and of statistics of excursions from sets. A few authors \cite {Kle89,Ker92,DKW13} extended Lamperti's result under weaker moment conditions, but only for the convergence of finite dimensional marginals  and not for functional scaling limits.

Let $R_{1/s}^{(\nu)}$ be a Bessel process with index $\nu$ (or equivalently of dimension $ \delta=2(\nu+1)$) started from $1/s$ (we refer to \cite[Chap.~XI]{RY99} for background on Bessel processes). By standard properties of Bessel processes, $ R_{1/s}^{(\nu)}$ does not touch $0$ for $r \leq  -1$, is reflected at $0$ for $ -1 < r < 1$, and absorbed at $0$ for $r \geq 1$.

In the particular case of Markov chains with asymptotically zero drifts satisfying \eqref{eq:zd1} \& \eqref{eq:zd2}, our main results specialize as follows:

\begin {thm}\label {thm:zd}Assume that \eqref{eq:zd1} \& \eqref{eq:zd2} hold.   
\begin{enumerate}
\item[(i)] If either $r \leq - 1$, or $r>1$, then we have
$$\left( \frac{ {X} _{n}( \lfloor n^{2} t \rfloor)}{n} ; t \geq 0 \right)   \quad\mathop{\longrightarrow}^{(d)}_{n \rightarrow \infty} \quad s R_{1/s}^{(\nu)}$$
in $\D(\R_{+},\R)$. 
\item[(ii)] If $ r>-1$, there exists an integer $K \geq 1$ such that $ \{ 1,2, \ldots, K\}$ is accessible by $X_{n}$ for every $n \geq 1$, and the following distributional convergence holds in $\D(\R_{+},\R)$:
$$\left( \frac{ {X}^{\dagger} _{n}( \lfloor n^{2} t \rfloor)}{n} ; t \geq 0 \right)   \quad\mathop{\longrightarrow}^{(d)}_{n \rightarrow \infty} \quad s R_{1/s}^{(\nu),\dagger},$$
where $ {X}^{\dagger} _{n}$ denotes the Markov chain $X_{n}$ stopped as soon as it hits $ \{ 1,2,  \ldots, K\} $ and  $R_{1/s}^{(\nu),\dagger}$ denotes the Bessel process $R_{1/s}^{(\nu)}$ stopped as soon as it hits $0$.  

In addition, if $A_{n}$ denotes the first time $X_{n}$ hits $ \{ 1,2, \ldots,K\} $, then
\begin{equation}
\label{eq:thm52}\frac{A_{n}}{n^{2}}  \quad\mathop{\longrightarrow}^{(d)}_{n \rightarrow \infty} \quad \frac{1}{2 s^{2} \cdot \gamma_{(1+r)/2}},
\end{equation}
where $\gamma_{(1+r)/2}$ is a Gamma random variable with parameter $(1+r)/2$.
\item[(iii)] If $r>1$, we have further
\begin{equation}
\label{eq:thm53}\frac{\Es{A_{n}^{q}}}{n^{2q}}  \quad\mathop{\longrightarrow}_{n \rightarrow \infty} \quad \frac{1}{{(2s^{2})^{q}}} \cdot \frac{ \Gamma \left(  \frac{1+r}{2}-q \right) }{\Gamma \left(  \frac{1+r}{2} \right) }
\end{equation}
for every $1 \leq q< {(1+r)/2}$. In particular,
$$ \frac{\Es{A_{n}}}{n^{2}}  \quad\mathop{\longrightarrow}_{n \rightarrow \infty} \quad \frac{1}{s^{2}(r-1)}.$$
\end{enumerate}
\end {thm}

These results concerning the asymptotic scaled functional behavior of Markov chains with asymptotically zero drifts and the fact that the scaling limit of the first time they hit $0$ is a multiple of an inverse gamma random variable may be new. We stress that the appearance of the inverse gamma distribution in this framework is related to a well-known result of Dufresne \cite{Dufresne}, see also the discussion in \cite{BY05} for further references. 

The main step to prove Theorem \ref {thm:zd} is to check that the conditions \eqref{eq:zd1} \& \eqref{eq:zd2} imply our assumptions introduced in Section \ref {sec:description} are satisfied:

\begin {prop}\label {prop:zd}Assertion \textbf{(A1)} holds with $a_{n}=n^{2}$ and $ \Pi=0$; Assertion \textbf{(A2)} holds with $b= \frac{2c-s^{2}}{2}$ and $ \sigma^{2}=s^{2}$; Assertion \textbf{(A3)} holds for every $ \beta>0$. Finally, if $r>1$, then Assertion \textbf{(A5)} holds for every $ \beta_{0} \in (2,1+r)$.\end {prop}

Before proving this, let us explain how to deduce Theorem \ref {thm:zd}  from Proposition \ref {prop:zd}.

\begin {proof}[Proof of  {Theorem \ref {thm:zd}}] By Proposition \ref {prop:zd}, for every $t \geq 0$, we have  $ \xi(t)= s B_{t}+\frac{2c-s^{2}}{2}t$ where $B$ is a standard Brownian motion. Note that $ \xi$ drifts to $- \infty$ if and only if $2c-s^{2}<0$, that is $r>-1$. By \cite[p.~452]{RY99}, $Y(t/s^{2})$ is a Bessel process $R^{(\nu)}_{1}$ with index $ \nu$ and dimension $ \delta$ given by $$ \nu  \coloneqq  \frac{2c-s^{2}}{2s^{2}}= - \frac{1+r}{2}, \qquad \delta  \coloneqq  1-r$$ started from $1$ and stopped as soon as its hits $0$. Hence by scaling, we can write $Y(t)= s R^{(\nu)}_{1/s}(t)$. Theorem \ref {thm:zd} then follows from Theorems \ref{thm:main1}, \ref {thm:main2}, \ref {thm:absorption} and \ref {thm:main3} as well as Remark \ref {rem:moments}. For \eqref{eq:thm52} and \eqref{eq:thm53}, we also use the fact that (see e.g.~\cite[p.~452]{RY99})
$$ \int_{0}^{\infty}e^{2(s B_{u}+\frac{2c-s^{2}}{2}u)} du  \quad\mathop{=}^{(d)} \quad  \frac{1}{2 s^{2} \cdot \gamma_{(1+r)/2}}.$$
This completes the proof.
\end {proof}

The proof of Proposition \ref {prop:zd} is slightly technical, and we start with a couple of preparatory lemmas.

\begin {lem}\label {lem:zd1} We have
$$ n^{2} \int_{|x|>1} \left| e^{x}-1 \right| \Pi_{n}^{\ast}(dx)  \quad\mathop{\longrightarrow}_{n \rightarrow \infty} \quad 0, \qquad  n^{2} \int_{|x|>1} \left( e^{x}-1 \right)^{2} \Pi_{n}^{\ast}(dx)  \quad\mathop{\longrightarrow}_{n \rightarrow \infty} \quad 0.$$
\end {lem}

\begin {proof}It is enough to establish the second convergence, which implies the first one. We show that $n^{2} \int_{1}^{\infty}(e^{x}-1)^{2} \Pi_{n}^{\ast}(dx) \rightarrow 0$ as $n \rightarrow \infty$ (the case when $x<-1$ is similar, and left to the reader). We write
\begin{eqnarray*}
n^{2} \int_{1}^{\infty}(e^{x}-1)^{2} \Pi_{n}^{\ast}(dx)&=&n^{2} \sum_{k \geq e n} \left(  \frac{k}{n}-1 \right)^{2}  p_{n,k}  =  \sum_{k \geq en} (k-n)^{p} \cdot \frac{1}{(k-n)^{p-2}} p_{n,k} \\
& \leq&  \sum_{k \geq en} (k-n)^{p} \cdot \frac{1}{(e-1)^{p-2} \cdot n^{p-2}} p_{n,k} = \frac{ \Es {\Delta_{n} \mathbbm {1}_{ \{ \Delta_{n} \geq (e-1){n}\} }}}{(e-1)^{p-2}n^{p-2}} \\
& \leq & \frac{C}{(e-1)^{p-2}} \cdot n^{-\delta} \qquad \textrm {by } \eqref{eq:zd1} ,
\end{eqnarray*}
which tends to $0$ as $n \rightarrow \infty$.
\end {proof}

\begin {lem}\label {lem:zd2} We have
$$ \lim_{\epsilon \rightarrow 0} \limsup_{n \rightarrow \infty} n^{2} \int_{|x| < \epsilon} |x|^{3} \cdot \Pi_{n}^{\ast}(dx)  \quad\mathop{=} \quad 0.$$
\end {lem}

\begin {proof}To simplify notation, we establish the result with $ \epsilon$ replaced by $ \ln(1+ \epsilon)$, with $ \epsilon \in (0,1)$. Write 
$$ \int_{|x| < \ln(1+\epsilon)} |x|^{3} \cdot \Pi_{n}^{\ast}(dx) =  \sum_{ (1+\epsilon)^{-1} \leq k/n \leq 1+ \epsilon}  \left|\ln \left( 1+ \frac{k-n}{n} \right)  \right|^{3} p_{n,k}.$$
But $ (1+\epsilon)^{-1} \leq k/n \leq 1+ \epsilon$ implies that $|k-n|/n \leq  \epsilon$, and there exists a constant $C'>0$ such that $ |\ln(1+x)|^{3} \leq C' |x|^{3}$ for every $|x| \leq 1$. Hence
$$ n^{2} \int_{|x| < \epsilon} |x|^{3} \cdot \Pi_{n}^{\ast}(dx)  \leq   \sum_{ (1+\epsilon)^{-1} \leq k/n \leq 1+ \epsilon}  |k-n|^{2}  \cdot \frac{|k-n|}{n} p_{n,k} \leq  \epsilon  \cdot \Es { \Delta_{n}^{2}},$$
and the result follows by \eqref{eq:zd2}. 
\end  {proof}

We are now in position to establish Proposition \ref {prop:zd}.

\begin {proof}[Proof of Proposition \ref {prop:zd}]
In order to check \textbf{(A1)}, we show that $ n^{2} \cdot \Pi_{n}^{\ast}([ln(a), \infty)) \rightarrow 0$ as $n \rightarrow \infty$ for every fixed $a>1$ (the proof is similar for $a \in (0,1)$) by writing that
$$ (a-1)^{p} n^{p} \sum_{k \geq an}p_{n,k} \leq  \sum_{k \geq an} (k-n)^{p} p_{n,k} \leq  \Es {  |\Delta_{n}|^{p}} \leq  C \cdot n^{p-2- \delta}.$$
Therefore,
$$n^{2} \cdot \Pi_{n}^{\ast}([ln(a), \infty)= n^{2} \sum_{k \geq a n} p_{n,k}  \leq  \frac{C}{(a-1)^{p}} n^{-\delta}\quad\mathop{\longrightarrow}_{n \rightarrow \infty} \quad 0.$$

To prove \textbf{(A2)}, we first show that
\begin{equation}
\label{eq:zdp1}n^{2} \int_{-1}^{1} \left( e^{x}-1-x- \frac{x^{2}}{2} \right)  \Pi_{n}^{\ast}(dx)  \quad\mathop{\longrightarrow}_{n \rightarrow \infty} \quad  0.
\end{equation}
Since there exists a constant $C'>0$ such that $| e^{x}-1-x- {x^{2}}/{2}| \leq C' x^{3}$ for every $|x| \leq 1$, for fixed $ \epsilon >0$, by Lemma \ref {lem:zd2} we may find $\eta>0$ such that
$$n^{2} \int_{-\eta}^{\eta} \left( e^{x}-1-x- \frac{x^{2}}{2} \right)  \Pi_{n}^{\ast}(dx) \leq  \epsilon$$
for every $n$ sufficiently large. But
$$n^{2} \int_{ \eta<|x|< 1} \left( e^{x}-1-x- \frac{x^{2}}{2} \right)  \Pi_{n}^{\ast}(dx)  \quad\mathop{\longrightarrow}_{n \rightarrow \infty} \quad 0$$
by the first paragraph of the proof. This establishes \eqref{eq:zdp1}. One similarly shows that
\begin{equation}
\label{eq:zdp2}n^{2} \int_{-1}^{1}  \left( (e^{x}-1)^{2}-x^{2} \right)  \Pi_{n}^{\ast}(dx)  \quad\mathop{\longrightarrow}_{n \rightarrow \infty} \quad 0.
\end{equation}

  Next observe that
$$n^{2} \int_{-\infty}^{\infty}(e^{x}-1) \Pi_{n}^{\ast}(dx)=n \Es { \Delta_{n}}\quad \hbox{and} \quad n^{2} \int_{-\infty}^{\infty}(e^{x}-1)^2 \Pi_{n}^{\ast}(dx)= \Es { \Delta^2_{n}}.$$
Thus, by Lemma \ref {lem:zd1} and \eqref{eq:zd2}, we have 
$$n^{2} \int_{-1}^{1}(e^{x}-1) \Pi_{n}^{\ast}(dx)  \quad\mathop{\longrightarrow}_{n \rightarrow \infty} \quad c, \qquad n^{2} \int_{-1}^{1}(e^{x}-1)^{2} \Pi_{n}^{\ast}(dx)  \quad\mathop{\longrightarrow}_{n \rightarrow \infty} \quad s^{2}.$$
Then write 
$$n^{2} \int_{-1}^{1}(e^{x}-1) \Pi_{n}^{\ast}(dx)=n^{2} \int_{-1}^{1}x\Pi_{n}^{\ast}(dx)+n^{2} \int_{-1}^{1}\frac{x^{2}}{2} \Pi_{n}^{\ast}(dx)+ n^{2} \int_{-1}^{1} \left( e^{x}-1-x- \frac{x^{2}}{2} \right)  \Pi_{n}^{\ast}(dx)$$
and
$$n^{2} \int_{-1}^{1}(e^{x}-1)^{2} \Pi_{n}^{\ast}(dx)= n^{2} \int_{-1}^{1} x^{2} \Pi_{n}^{\ast}(dx)+  n^{2} \int_{-1}^{1}  \left( (e^{x}-1)^{2}-x^{2} \right)  \Pi_{n}^{\ast}(dx).$$
By \eqref{eq:zdp1} and \eqref{eq:zdp2}, the last term of the right-hand side of the two previous equalities tends to $0$ as $ n \rightarrow \infty$. It follows that
$$ n^{2} \cdot  \int_{-1}^{1}x \ \Pi_{n}^{\ast}(dx)  \quad\mathop{\longrightarrow}_{n \rightarrow \infty} \quad b, \qquad n^{2} \cdot \int_{-1}^{1}x^{2} \ \Pi_{n}^{\ast}(dx)  \quad\mathop{\longrightarrow}_{n \rightarrow \infty} \quad \sigma^{2},$$
  where $b$ and $ \sigma^{2}$ satisfy
  $$c=b+ \frac{ \sigma^{2}}{2}, \qquad s^{2}= \sigma^{2}.$$
  This shows that \textbf{(A2)} holds.
  
In order to establish that \textbf{(A3)} holds for every $ \beta_{0} \in [0,p]$, first note that the constraint on $ \beta_{0}$ yields the existence of  a constant $C'>0$ such that $k^{\beta_{0}}/(k-n)^{p} \leq C' n^{ \beta_{0}-p}$ for every $k \geq en$ and $n \geq 1$. Then write
\begin{eqnarray*}
n^{2} \cdot \int_{1}^{\infty}  e^{\beta_{0} x }\ \Pi^{\ast}_{n}(dx) = n^{2-\beta_{0}} \sum_{k \geq  e n}k^{\beta_{0}} p_{n,k} &=& n^{2-\beta_{0}} \sum_{k \geq  e n}(k-n)^{p} \frac{k^{\beta_{0}}}{(k-n)^{p}} p_{n,k} \\
&\leq&  C' n^{2-p}  \sum_{k \geq  e n}(k-n)^{p} p_{n,k} \leq CC'n^{-\delta}.
\end{eqnarray*}
This shows that \textbf{(A3)} holds.

Finally, for the last assertion of Proposition \ref{prop:zd} observe that $ \Psi(\lambda)= \frac{1}{2} s^{2} \lambda^{2}+ \frac{2c-s^{2}}{2} \lambda$, so that $ \Psi(2)=2c+s^{2}$ and $\Psi(1+r)=0$. In particular, if $2c+s^{2}<0$, one may find $ \beta_{0} \in (2, 1+r)$ such that $ \Psi(\beta_{0})<0$.  This shows \textbf{(A4)}. Finally, for \textbf{(A5)}, note that $\Es{|X_{n}(1)-n|^{\beta_{0}}}=\Es{|\Delta_{n}|^{\beta_{0}}} \leq \Es{|\Delta_{n}|^{p}}<\infty$ implies that $\Es{X_{n}(1)^{\beta_{0}}}<\infty$. This completes the proof.\end {proof}

\begin {rem}The results of \cite {HMW13} establish many estimates concerning  various statistics of excursions of $X_{1}$ from $1$ (such as the duration of the excursion, its maximum, etc.). Unfortunately, those estimates are not enough to establish directly \eqref{eq:mqthm3bis},\eqref{eq:thm52} and \eqref{eq:thm53}.  However, only in the particular case of Bessel-type random walks, it is possible to use the local estimates of \cite {Ale11} in order to establish \eqref{eq:mqthm3bis}, \eqref{eq:thm52} and \eqref{eq:thm53}  directly. 
\end {rem}

\subsection{The number of fragments in a fragmentation-coagulation process}
Exchangeable fragmentation-coalescence processes were introduced by J.~Berestycki \cite{Ber04}, as Markovian models whose evolution combines the dynamics of exchangeable coalescent processes and those of homogeneous fragmentations. The fragmentation-coagulation process that we shall consider in this Section can be viewed as a special case in this family.

Imagine a particle system in which particles may split or coagulate as time passes.
For the sake of simplicity, we shall focus on the case when coalescent events are simple, that is the coalescent dynamics is that of a  $\Lambda$-coalescent in the sense of Pitman \cite{Pit99}. Specifically, $\Lambda$ is a finite measure on $[0,1]$; we shall implicitly assume that $\Lambda$ has no atom at $0$, viz.~$\Lambda(\{0\})=0$.
In turn, we suppose that the fragmentation dynamics are homogeneous (i.e.~independent of the masses of the particles) and governed by a finite dislocation measure which  only charges mass-partitions having a finite number (at least two) of components. That is, almost-surely, when a dislocation occurs, the particle which splits is replaced by a finite number of smaller particles. 

The process $\#_n=(\#_n(t);  t\geq 0)$ which counts the number of particles as time passes, when the process starts at time $t=0$ with $n$ particles,  is a continuous-time Markov chain with values in $\N$. 
More precisely, the rate at which $\#_n$ jumps from $n$ to $k<n$ as the result of a simple coagulation event involving $n-k+1$ particles is given by
$$g_{n,k}=\int_{(0,1]} 
{n \choose k-1}
 x^{n-k-1}(1-x)^{k-1}\Lambda(dx).$$
We also write 
$$g_n= \sum_{k=1}^{n-1} g_{n,k}= \int_{(0,1]}\left(1-(1-x)^n-n x (1-x)^{n-1}\right) x^{-2}Ê\Lambda(dx)$$
for the total rate of coalescence.
In turn, let $\mu$ denote a finite measure on $\N$, such that the rate at which each particle splits into $j+1$ particles (whence inducing an increase of $j$ units for the number of particles) when a  dislocation event occurs, is given by $\mu(j)$ for every $j\in\N$.

We are interested in  the jump chain $X_n=(X_n(k) ; k\geq 0)$ of $\#_n$, that is the discrete-time embedded Markov chain of the successive values taken by $\#_n$. The transition probabilities $p_{n,k}$ of $X_n$ are thus given by
$$p_{n,k}=  \left\{ \begin{matrix} n\mu(k-n)/(g_n+n\mu(\N)) & \hbox{ for } &k>n,\\
  g_{n,k}/(g_n+n\mu(\N)) & \hbox{ for }& k<n.
 \end{matrix} \right.$$
 We assume from now on that the measure $\mu$ has a finite mean 
$$m\coloneqq\sum_{j=1}^{\infty} j \mu(j)<\infty$$
and further that
$$\int_{(0,1]} x^{-1} \Lambda(dx)<\infty\,.$$
  Before  stating our main result about the scaling limit of the chain $X_n$, it is convenient 
 introduce the measure $\Pi(dy)$ on $(-\infty,0)$  induced by the image of $x^{-2}\Lambda(dx)$ by the map $x\mapsto y = \ln(1-x)$ and observe that 
 $$\int_{(-\infty,0)} (1\wedge |y|) \ \Pi(dy)<\infty.$$
We may thus consider  the spectrally negative L\'evy process $\xi=(\xi(t), t\geq 0)$ whose Laplace transform given by
\begin{eqnarray*}
\E {\exp(q\xi(t))} &=&\exp\left( \frac{t}{\mu(\N)}\left (mq + \int_{(-\infty,0)}(e^{qy}-1) \ \Pi(dy)\right)\right )  \\
& =& \exp\left( \frac{t}{\mu(\N)}\left (mq + \int_{(0,1)}((1-x)^{q}-1) \cdot x^{-2}\Lambda(dx)\right)\right ) .
\end{eqnarray*}
We point out that $\xi$ has finite variations, more precisely it is the sum of the negative of a subordinator and  a positive drift, and also that $\xi$ drifts to $+\infty$, oscillates, or drifts to $-\infty$ according as the mean
$$\Es{\xi_1}=m+\int_{(-\infty,0)}y \ \Pi(dy) = m+\int_{0}^{1} \frac{ \ln(1-x)}{x^{2}} \ \Lambda(dx)$$
is respectively strictly positive, zero, or strictly negative (possibly $-\infty$).

\begin{cor} \label{cor:C2}  Let $ (Y(t), t\geq 0)$ denote the positive self-similar Markov process with index $1$, which is associated via Lamperti's transform to the spectrally negative L\'evy process $\xi$. 
\begin{enumerate}
\item[(i)] 
If $\xi$ drifts to $+\infty$ or oscillates, then there is the weak convergence  in $\D(\R_{+},\R)$
$$\left(  \frac{X_n(\lfloor nt\rfloor)}{n}; t\geq 0\right) \quad\mathop{\longrightarrow}^{(d)}_{n \rightarrow \infty} \quad (Y(t);  t\geq 0).$$

\item[(ii)]  If $\xi$ drifts to $-\infty$, then $A^{(1)}_n=\inf\{k\geq 1: X_n(k)=1\}$ is a.s.~finite for all $n\geq 1$, 
$$ \frac{A^{(1)}_n}{n} \quad\mathop{\longrightarrow}^{(d)}_{n \rightarrow \infty} \quad \int_0^{\infty}Êe^{\xi(s)} ds\,,$$
and this weak convergence holds jointly with
$$\left(  \frac{X_n \left( \lfloor nt\rfloor \wedge A^{(1)}_n \right) }{n} ; t\geq 0\right) \quad\mathop{\longrightarrow}^{(d)}_{n \rightarrow \infty} \quad (Y(t) ; t\geq 0)$$
in $\D(\R_{+},\R)$. 

\item[(iii)]  If $m < \int_{(-\infty,0)}(1-e^{y})\ \Pi(dy) = \int_{0}^{1} x^{-1} \Lambda(dx)$ 
and $\sum_{j=1}^{\infty} j^{\beta} \mu(j)<\infty$ for some $\beta>1$, then $\xi$ drifts to $-\infty$ and 
$$\left( \frac{X_n(\lfloor nt\rfloor)}{n}; t\geq 0\right) \quad\mathop{\longrightarrow}^{(d)}_{n \rightarrow \infty} \quad (Y(t); t\geq 0)$$
in $\D(\R_{+},\R)$. 
In addition, for every $1 \leq p < \beta$ such that $m < \int_{(0,1)}(1-(1-x)^{p})/p \cdot x^{-2}\Lambda(dx)$, we have
$$ \Es{ \left( \frac{A^{(1)}_n}{n} \right) ^{p}} \quad\mathop{\longrightarrow}^{(d)}_{n \rightarrow \infty} \quad  \Es{ \left( \int_0^{\infty}Êe^{\xi(s)} ds \right)^{p}}.$$
 \end{enumerate}
\end{cor}

\begin{proof}  We first note that, since $\mu$ as finite mean $m$, 
$$\lim_{n\to\infty} n\sum_{k=n}^{\infty}\left(f(k/n)-f(1)\right) \mu(k-n)
= mf'(1)$$
for every bounded function $f: \R_+\to \R$  that is differentiable at $1$.
We also  lift from Lemma 9 from Haas \& Miermont \cite{HM11}  that in this situation
$$\lim_{n\to\infty} \sum_{k=1}^{n-1}\left(f(k/n)-f(1)\right) g_{n,k}
= \int_{(0,1]}\left(f(1-x)-f(1)\right) x^{-2} \ \Lambda(dx).
$$

Then we  observe that there is the identity
$$\frac{g_n}{n}=  \int_{(0,1]}n^{-1}\left(\sum_{j=0}^{n-2}((1-x)^j-(1-x)^{n-1})\right) x^{-1} \Lambda(dx).$$
It follows readily by dominated convergence from our assumption $\int_{(0,1]} x^{-1} \Lambda(dx)<\infty$ that $g_n=o(n)$, and therefore $g_n+n\mu(\N)\sim n\mu(\N)$ as $n \rightarrow \infty$.
Hence 
$$\lim_{n\to\infty} n \mu(\N) \sum_{k=1}^{\infty} (f(k/n)-f(1))p_{n,k} = m f'(1)+ \int_{(0,1]}\left(f(1-x)-f(1)\right) \ x^{-2}\Lambda(dx),
$$
and for every bounded function $h:\R\to \R$ which is differentiable at $0$, we therefore have 
\begin{eqnarray*}
\lim_{n\to\infty} n\int_{\R}\left(h(x)-h(0)\right) \ \Pi^\ast_n(dx) &=& 
\lim_{n\to\infty} n  \sum_{k=1}^{\infty} (h(\ln(k/n))-h(0))p_{n,k}\\
&=&\frac{1}{\mu(\N)}\left(m h'(0)+ \int_{(0,1]}\left(h(\ln (1-x))-h(0)\right) x^{-2}\Lambda(dx)\right )\\
&=&\frac{1}{\mu(\N)}\left(m h'(0)+ \int_{(-\infty,0)}\left(h(y))-h(0)\right) \Pi(d y)\right )\\
\end{eqnarray*}
where $\Pi(dy)$ stands for the image of $x^{-2}\Lambda(dx)$ by the map $x\mapsto y = \ln(1-x)$.  This proves that the assumptions \textbf{(A1)} and \textbf{(A2)} hold (with $ \Pi/ \mu(\N)$ instead of $ \Pi$ to be precise), and then (i) follows from Theorem \ref{thm:main1}. 
Note also that 
\begin{equation}
\label{eq:adapt}n \cdot \int_{1}^{\infty }e^{\beta x} \ \Pi^\ast_n(dx) =
 n \cdot  \sum_{k>e n }  \frac{k}{n}  \cdot p_{n,k} \leq  \frac{1}{ \mu(\N)} \sum_{k>(1-e)n}  (1+k) \mu(k) \leq (1+ m)/ \mu(\N),
\end{equation}
which shows that \textbf{(A3)} is fulfilled. Hence (ii) follows from Theorems \ref{thm:main2} and \ref{thm:absorption} (i).

Finally, it is easy to check that when the assumptions of (iii) are fulfilled, then \textbf{(A4)} and  \textbf{(A5)} hold. Indeed, as for \eqref{eq:adapt}, for every $ \beta>0$ we have
$$n \cdot \int_{1}^{\infty }e^{\beta x} \ \Pi^\ast_n(dx) =
 n \cdot  \sum_{k>e n } \left( \frac{k}{n} \right) ^{\beta}p_{n,k} \leq  \frac{n^{1-\beta}}{ \mu(\N)} \sum_{k>(1-e)n}  k^{\beta} \mu(k),$$
and we can thus invoke Theorem \ref{thm:main3}, as well as Theorem \ref {thm:absorption} (ii) and Remark \ref {rem:moments}.
\end{proof}

Roughly speaking, Corollary \ref {cor:C2} tells us that in case (i), the number of blocks drifts to $+ \infty$ and in case (iii), once the number of blocks is of order $o(n)$, it will remain of order $o(n)$ on time scales of order $a_{n}$. In case (ii), we are only able to understand what happens until the moment when there is only one block. It is plausible that in some cases, the process counting the number of blocks may then ``restart'' (see Section \ref {sec:questions} for a similar discussion).

\section {Open questions}
\label {sec:questions}

Here we gather some open questions.

\begin {ques}\label {q:1}Is it true that Theorem \ref {thm:main2} remains valid if \textbf{(A3)} is replaced with the  condition $ \inf \{ i \geq 1; X_{n}(i) \leq K \}< \infty$ almost surely for every $n \geq 1$?
\end {ques}

\begin {ques}\label {q:2}Is it true that Theorem \ref {thm:main3} remains valid if \textbf{(A4)} is replaced with the  condition that $ \Es{ \inf \{ i \geq 1; X_{n}(i) \leq K \}}< \infty$  for every $n \geq 1$?
\end {ques}

It seems that answering Questions \ref {q:1} and \ref {q:2} would require new techniques which are not based on Foster--Lyapounov type estimates. Unfortunately, up to now, even in the case of Markov chains with asymptotically zero drifts, all refined analysis is based on such estimates.

A first step would be to answer these questions in the particular case Markov chains with asymptotically zero drifts; recalling the notation of Section \ref {sec:zd}:

\begin {ques}\label {q:3}Consider a Markov chain with asymptotically zero drifts satisfying \eqref{eq:zd2} only. Under what conditions do we have
$$\left( \frac{ {X} _{n}( \lfloor n^{2} t \rfloor)}{n} ; t \geq 0 \right)   \quad\mathop{\longrightarrow}^{(d)}_{n \rightarrow \infty} \quad s R_{1/s}^{(\nu)} \quad ?$$
\end {ques} 
When  in addition the assumption \eqref{eq:zd1} is  satisfied, our results settle the cases $r \leq -1$ and $r>1$. Also, as it was already mentioned, using moment methods, Lamperti \cite[Theorem 5.1]{Lam62b} settles the case $r \leq 1$ under the assumptions  that  $\sup_{n \geq 1} \Es {|\Delta_{n}|^{4}}< \infty$  and that Markov chain is uniformly null (see \cite {Lam62b} for a definition). We mention that if $X_{n}$ is irreducible and not positive recurrent (which is the case when $r<1$), then it is uniformly null. 

However, in general, the asymptotic behavior of $X_{n}$ will be very sensitive to the laws of $X_{k}(1)$ for small values of $k$. For example, even in the Bessel-like random walk case, one drastically changes the behavior of $X_{n}$ just by changing the distribution of $X_{1}(1)$ in such a way that $ \Es {X_{1}(1)^{2}}= \infty$. More generally:

\begin {ques}\label {q:4}Assume that \textbf{(A1)} and \textbf{(A2)} hold, and that there exists an integer $1 \leq n \leq K$ such that  $ \Es{ \inf \{ i \geq 1; X_{n}(i) \leq K \}}= \infty$. Under what conditions on the probability distributions $X_{1}(1),X_{2}(1), \ldots,$ $X_{K}(1)$ does the Markov chain $X_{n}$ have a continuous scaling limit (in which case $0$ is a continuously reflecting boundary)? A discontinuous c\`adl\`ag scaling limit (in which case $0$ is a discontinuously reflecting boundary)?
\end {ques}

As a first step, one could first try to answer this question under the assumptions \textbf{(A3)} or \textbf{(A4)} which enable the use of Foster--Lyapounov type techniques. 
We intend  to develop this in a future work.

\begin {ques}Assume that $ \xi$ does not drift to $- \infty$, and that if $ \P_{x}$ denotes the law of $Y$ started from $x>0$, then $ \P_{x}$ converges weakly as $x \downarrow 0$ to a probability distribution denoted by $ \P_{0}$. Does there exist a family $(p_{n,k})$ such that the law of $Y$ under $ \P_{0}$ is the scaling limit of $X_{n}$ as $n \rightarrow \infty$? If so, can one find sufficient conditions guaranteeing this distributional convergence?
\end {ques}

\begin {ques}Assume that  $\xi$ drifts to $-\infty$, so that $Y$ is absorbed at $0$. Assume that $Y$ has a recurrent extension at $0$. Does there exist a family $(p_{n,k})$ such that this recurrent extension is the scaling limit of $X_{n}$ as $n \rightarrow \infty$?  If so, can one find sufficient conditions guaranteeing this distributional convergence?
\end {ques}


\end{document}